\documentclass{amsart} 
\usepackage[colorlinks]{hyperref}
\usepackage{enumerate} 
\usepackage{upref}
\usepackage{verbatim} 
\usepackage{color}
\usepackage{graphicx}
\usepackage{todonotes}
\newcommand{\sign}{\mathop{\rm sign}}
\newcommand*{\mailto}[1]{\href{mailto:#1}{\nolinkurl{#1}}}
\usepackage[latin1]{inputenc}
\usepackage{amssymb, amsmath, amsfonts, amsthm}
\usepackage[all]{xy}

\usepackage{marginnote}

\DeclareMathOperator{\id}{Id}
\DeclareMathOperator{\meas}{meas}
\DeclareMathOperator{\supp}{supp}
\newcommand{\dott}{\, \cdot\,}

\newcommand{\Gr}{G}

\newcommand{\D}{\ensuremath{\mathcal{D}}}

\newcommand{\F}{\ensuremath{\mathcal{F}}}

\newcommand{\abs}[1]{\left\vert#1\right\vert}

\newcommand{\Real}{\mathbb R}

\newcommand{\Natural}{\mathbb N}
\newcommand{\norm}[1]{\left\lVert#1\right\rVert}
\newcommand{\bnorm}[1]{\bigl\lVert#1\bigr\rVert}

\newcommand{\Ltwo}{{L^2(\Real)}}

\newcommand{\Linf}{{L^\infty(\Real)}}

\newcommand{\muac}{\mu_{\text{\rm ac}}}

\newcommand{\nn}{\nonumber}

\newcommand{\dx}{\,dx}

\newtheorem{theorem}{Theorem}[section]

\newtheorem{definition}[theorem]{Definition}

\newtheorem{remark}[theorem]{Remark}
\numberwithin{equation}{section}

\allowdisplaybreaks

\begin{document}

\title[The 2CH system]{On the equivalence of Eulerian and Lagrangian variables for the  two-component Camassa--Holm system}

\author[M. Grasmair]{Markus Grasmair}
\address{Department of Mathematical Sciences\\ NTNU Norwegian University of Science and Technology\\ NO-7491 Trondheim\\ Norway}
\email{\mailto{markus.grasmair@ntnu.no}}

\author[K. Grunert]{Katrin Grunert}
\address{Department of Mathematical Sciences\\ NTNU Norwegian University of Science and Technology\\ NO-7491 Trondheim\\ Norway}
\email{\mailto{katrin.grunert@ntnu.no}}
\urladdr{\url{http://www.math.ntnu.no/~katring/}}

\author[H. Holden]{Helge Holden}
\address{Department of Mathematical Sciences\\ NTNU 
  Norwegian University of Science and Technology\\
  NO-7491 Trondheim\\ Norway}
\email{\mailto{helge.holden@ntnu.no}}
\urladdr{\url{http://www.math.ntnu.no/~holden/}}

\date{\today} 
\thanks{Research supported in part by the Research Council of Norway projects NoPiMa and WaNP, and by the Austrian Science Fund (FWF) under Grant No.~J3147.  KG and HH are grateful to Institut Mittag-Leffler, Stockholm, for the generous hospitality during the fall of 2016, when part of this paper was written.}  

\begin{abstract}
The  Camassa--Holm equation and its two-component Camassa--Holm system generalization both experience wave breaking in finite time. To analyze this, and to obtain solutions past wave breaking, it is common to reformulate the original equation given in Eulerian coordinates, into a system of ordinary differential equations in Lagrangian coordinates. It is of considerable interest to study the stability of solutions and how this is manifested in Eulerian and Lagrangian variables.  We identify criteria of convergence, such that convergence in Eulerian coordinates is equivalent to convergence in Lagrangian coordinates. In addition, we show how one can approximate global conservative solutions of the scalar Camassa--Holm equation by smooth solutions of the two-component Camassa--Holm system that do not experience wave breaking.  
\end{abstract}

\maketitle

\section{Introduction}

The prevalent  way to analyze the ubiquitous wave breaking for the  Camassa--Holm (CH) equation, is to transform the original equation from its Eulerian variables into a new coordinate system, e.g.~in Lagrangian variables. The reason for the transformation is that while the solution develops 
singularities  in Eulerian coordinates, the solution remains smooth in the Lagrangian framework.  This invites the question of a closer analysis of the transformation between the Eulerian and the Lagrangian variables.  That is the goal of the present paper.

A  two-component generalization of the CH equation was introduced in \cite[Eq.~(43)]{OlverRosenau}, and we will study the above question in this setting. It turns out that this system, denoted the two-component Camassa--Holm (2CH) system, has a regularizing effect on the original CH equation as long as the density $\rho$ remains positive.   To set the stage, we recall that the 2CH system can be written as
\begin{subequations}
  \label{eq:rewchsys10}
  \begin{align}
    \label{eq:rewchsys11}
    u_t+uu_x+P_x&=0,\\
    \label{eq:rewchsys12}
    \rho_t+(u\rho)_x&=0,
  \end{align}
\end{subequations}
where $P$ is implicitly defined by
\begin{equation}
  \label{eq:rewchsys13}
  P-P_{xx}=u^2+\frac12u_x^2+\frac12\rho^2.
\end{equation}
The original CH equation \cite{CH,CHH} is the special case where $\rho$ vanishes identically. The CH equation possesses many intriguing properties, and the main challenge when one considers the Cauchy problem, is that the solution develops singularities in finite time, independent of the smoothness initially.    This singularity is characterized by the $H^1$-norm of the function $u$ remaining finite, while the spatial derivative $u_x$ goes to negative infinity at a specific point at the time of wave breaking. The structure of the points of wave breaking may be intricate \cite{grunert}.  The behavior in the proximity of the point of wave breaking, and, in particular, the prolongation of the solution past wave breaking, has been extensively studied.  See, e.g.,  
\cite{BC,BreCons:05a,ChenLiuZhang,ChenLiu,cons:98b,EscherLechtenfeldYin,FuQu,GHR2,GHR3,GHR5,GHR4,GHR,GuanKarlsenYin,GY, GuanYin2011, GuiLiu2010, GuiLiu2011, GuoZhou2010, HR,HRdiss,HRdissMP} and references therein.  The key point here is that past wave breaking uniqueness fails, and there is a continuum of distinct solutions \cite{GHR}, with two extreme solutions called dissipative and conservation solutions, respectively.  The various solutions can be characterized by the behavior of the total energy, as measured by the local $H^1$ density of the solution $u$. As mentioned above, the density $\rho$ has a regularizing effect on the solution: If $\rho$  is positive on the line initially, then the solution will not develop singularities \cite{consIvan,GHR4}.  A local result, saying that if  $\rho$ initially is smooth on an interval, then the solution will remain smooth on the interval determined by the characteristics emanating from the endpoints of the original interval, can be found in  \cite[Thm.~6.1]{GHR4}. This is surprising, as  the 2CH system has infinite speed of propagation \cite{henry}.

In this paper we study in detail the relation between the Eulerian and the Lagrangian variables, and, in particular, the stability of solutions in the two coordinate systems. Two aspects are considered. First one may ask if the solution of the 2CH system will converge to a solution of the CH equation in the limit when the density $\rho$ vanishes, and if so, to which of the plethora of solutions. This problem has also been studied in \cite{GHR4}. We show that the limit is the so-called conservative solution of the CH equation where the energy is preserved, see Theorem \ref{thm:timesmoothapprox}.  
The second question addresses the relation between stability in Eulerian variables and stability in Lagrangian variables in general.  The short answer is that the two notions are equivalent.  
This result can hardly be considered surprising.  However, as each of the norms for  the variables is rather intricate, and the relation between them is highly nonlinear, the actual proofs are considerably more technical than we expected.  In part, this is due to the fact that the solution does develop singularities in Eulerian coordinates, while it remains smooth in the Lagrangian framework. We have chosen to give rather detailed proofs, as we find that eases the understanding.  Each proof is broken down into  shorter technical arguments for the benefit of the reader.

Let us describe more precisely the content of this paper. A key role is played by the non-negative Radon measure $\mu$ with absolutely continuous part $\muac=(u_x^2+\bar\rho^2)dx$. Here $\rho-\bar\rho$ is a real constant, and $\bar\rho$ is square integrable. The dynamics between the singular and absolutely continuous part of the measure encode the wave breaking. In Section \ref{sec:2} we consider the Cauchy problem for the CH equation with initial data $(u,\mu)$.  
We mollify these data to obtain a sequence $(u_n,\rho_n,\mu_n)$ with positive density $\rho_n$. The main result in this section, 
Theorem~\ref{thm:approxEuler}, shows that indeed  $u_n\to u$ in $H^1$ while $\rho_n\rightharpoonup 0$,  and $\mu_n((-\infty,x]) \to \mu((-\infty,x])$ at points of continuity of the limit.  In Theorem~\ref{thm:timesmoothapprox} we prove that the same result applies to the solution of the initial value problem. More specifically, we show (in obvious notation) that the solution $(u_n(t),\rho_n(t),\mu_n(t))$ of \eqref{eq:rewchsys10} with initial data  $(u_{0,n},\rho_{0,n},\mu_{0,n})$ will converge to the conservative solution $(u(t),\mu(t))$ with initial data $(u_0,\mu_0)$.
In Section \ref{sec:3} we study how this approximation by a mollification procedure carries over in Lagrangian coordinates. To detail this, we first need to recall the transformation between Eulerian and Lagrangian variables. We are given the pair of functions 
$(u,\rho)\in H^1\times L^2$ (Eulerian variables). For simplicity we let $\rho=\bar\rho$. In addition, we need the energy density in the form of a positive Radon measure $\mu$, that was introduced above, such that the absolutely continuous part equals 
$\muac=(u_x^2+\rho^2) \dx$.  The characteristic is given by $y(\xi)=\sup\{y \mid \mu((-\infty,y))+y<\xi\}$. The Lagrangian velocity,  energy density, and density read $U=u(y)$,  $h=1-y_\xi$, and  $r=\rho(y)y_\xi$, respectively. The full set of Lagrangian variables is then $X=(y,U,h,r)$.  We write  $X=L((u,\rho,\mu))$, and $(u,\rho,\mu)=M(X)$.  There is a lack of uniqueness in this transformation, corresponding to the fact that a particle trajectory can be parametrized in several distinct ways. In our context we denote this by \textit{relabeling}.  Thus $M\circ L=\id$, while $L\circ M$ is only the identity on the equivalence classes of Lagrangian functions that correspond to one and the same Eulerian solution, see \cite[Thm.~3.12]{HR}.  We prove that the convergence   
 $(u_n,\rho_n,\mu_n)\to (u,0,\mu)$ implies that $X_n\to X$ (in obvious notation)  in the appropriate norm, see Theorem~\ref{thm:approxLagran}.  The proof is surprisingly intricate and applies the notion of relabeling. 
 
The situation is turned around in  Section \ref{sec:4}, where we consider an arbitrary sequence of Lagrangian coordinates  $X_n$ that converges to $X$, thus $X_n \to X$ in an appropriate norm. It is then shown that the corresponding Eulerian variables   $(u_n,\rho_n,\mu_n)$ converge to $(u,\rho, \mu)$, see 
Theorem \ref{thm:LagEuler}.  In Section \ref{sec:5} we study how general convergence in Eulerian coordinates carries over to Lagrangian variables. To be more specific, consider a sequence $(u_n,\rho_n,\mu_n)$ that converges to $(u,\rho, \mu)$. Then we show in Theorem~\ref{thm:EulerLag} that the corresponding Lagrangian coordinates converge. Here it is not assumed that the sequence $(u_n,\rho_n,\mu_n)$ is a mollification of $(u,\rho, \mu)$.
Finally, in Section \ref{sec:6} we consider the time-dependent case. Consider a sequence of initial data  $(u_{n,0},\rho_{n,0},\mu_{n,0})$ that converges to $(u_0,\rho_0,\mu_0)$ in $\D$.  In Theorem \ref{thm:timegeneralapprox} it is shown that the corresponding solutions converge for each fixed positive time.  The proof transfers the convergence issue from Eulerian variables to Lagrangian coordinates, analyzes it in these variables, and finally translates the result back to the original variables.

\section{Approximation in Eulerian coordinates} \label{sec:2}

The aim of this section is to show that any initial data $(u,0,\mu)$ of the CH equation can be approximated by a sequence of smooth initial data $(u_{n},\rho_{n},\mu_{n})$ of the 2CH system. We start by introducing the Banach spaces needed in this context, before recalling the definition of the set of Eulerian coordinates for the 2CH system (and hence also for the CH equation). Thereafter we state and prove the approximation theorem.

Let 
\begin{equation}
L^2_{\rm const}(\Real)= \{\rho\in L^1_{\rm loc}(\Real) \mid \rho(x)=\bar \rho(x)+k, \, \bar \rho\in L^2(\Real),  k\in \Real\}.\label{eq:Lconst} 
\end{equation}
Then we can associate to any $\rho\in L^2_{\rm const}(\Real)$ the unique pair $(\bar\rho, k)\in L^2(\Real)\times\Real$. Thus, if we equip $L^2_{\rm const}(\Real)$ with the norm 
\begin{equation}\label{normL2inf}
 \norm{\rho}_{L^2_{\rm const}}=\norm{\bar \rho}_{L^2}+\vert k\vert,
\end{equation}
then $L^2_{\rm const}(\Real)$ is a Banach space.   

We are now ready to introduce the set of Eulerian coordinates of the 2CH system (and hence also of the CH equation). The case of the CH equation corresponds to $\rho(x)=0$  for all $x\in\Real$. 

\begin{definition}[Eulerian coordinates]
 The set $\D$ is composed of all triples $(u,\rho,\mu)$ such that $u\in H^1(\Real)$, $\rho\in L^2_{\rm const}(\Real)$, and $\mu$ is a positive finite Radon measure whose absolutely continuous part $\mu_{\rm ac}$ satisfies
\begin{equation*}
 \mu_{\rm ac}=(u_x^2+\bar \rho^2)dx.
\end{equation*}
We write $F(x)=\mu((-\infty,x])$.
\end{definition}

 We will need a standard \textit{Friedrichs mollifier} $\phi\in C^\infty_c(\Real)$, 
chosen in such a way that $\phi(x)=\phi(-x)\geq 0$, $\norm{\phi}_{L^1}=1$, $\phi'(x)>0$ for $x\in(-1,0)$, and $\supp (\phi)=[-1,1]$. 

\begin{theorem}\label{thm:approxEuler}
Given $(u,0,\mu)\in\D$, let $(u_n,\rho_n,\mu_n)$ be given through
\begin{subequations}\label{smoothsequence}
\begin{align}\label{1:eq}
 u_n(x)& =\int_\Real n\phi(n(x-y))u(y)dy,\\ 
 \rho_n(x)&=\Big(\frac1{n^2}+\int_\Real n\phi'(y)F(x-\frac{y}{n})dy-\big(\int_\Real \phi(y)u_x(x-\frac{y}{n})dy\big)^2\Big)^{1/2}  \label{1:eqB}\\
                &= \frac1n +\bar\rho_n(x), \notag\\
 \mu_n(x)&=u_{n,x}^2(x)+\bar\rho_n^2(x). \label{1:eqC}
 \end{align}
 \end{subequations}
 Define moreover
 \begin{equation*}
   F_n(x)=\mu_n((-\infty,x]).
 \end{equation*}
 Then $(u_n,\rho_n,\mu_n)\in\D$ is a sequence of smooth functions,
 which approximates $(u,0,\mu)$ in the following sense:
 \begin{align*}
 u_n\to u &\quad \text{ in } H^1(\Real),\\
 F_n(x)\to F(x) &\quad \text{ for every } x \text{ at which } F \text{ is continuous.}
 \end{align*}
\end{theorem}

\begin{proof}
We split the proof into several steps.

{\bf Step 1.  Approximation of $u$ by smooth functions $u_n$.} 
By assumption we have $u\in H^1(\Real)$. Thus, application of Minkowski's
inequality for integrals and the dominated convergence theorem yield
that $u_n$ defined in \eqref{1:eq} converges to $u$ in
$H^1(\Real)$. Moreover, the smoothness of $\phi$ implies that $u_n \in
C^\infty(\Real)$.

{\bf Step 2. Construction of some auxiliary functions and measures.}

We start by defining the auxiliary function
\begin{equation}\label{def:hatFn}
\hat F_{n}(x)= \int_\Real n\phi(n(x-y))F(y)dy.
\end{equation}
Then $\hat F_n$ is smooth and converges pointwise to $F$ at
every point $x$ at which $F$ is continuous. Now recall that $F(x) =
\mu((-\infty,x])$ and denote by $\mu_{\rm d}$ the purely discrete part
of the finite Radon measure $\mu$. Then $\mu_{\rm d}$ can be written
as an at most countable sum of Dirac measures, the positions of which
coincide with the set of discontinuities of $F$. In particular, $F$ is
continuous almost everywhere, and thus $\hat{F}_n$ converges to $F$
pointwise almost everywhere.
Define moreover
\begin{equation}\label{def:hatmun}
\hat\mu_{n}(x)=\int_\Real n^2\phi'(n(x-y))F(y)dy=\int_\Real n\phi(n(x-y))d\mu(y) \ge 0. 
\end{equation}
Then we obtain by Fubini's theorem that $\norm{\hat
  \mu_n}_{L^1}=\norm{\mu}$ for all $n\in\Natural$.

As a next step, we will associate a sequence of densities $\hat\rho_n$ to  $(u_n, \hat\mu_n)$. To that end, we note, using the
Cauchy--Schwarz inequality and the fact that $\norm{\phi}_{L^1}  =1$,
that 
\begin{align*}
 \int_\Real \phi(z)&u_x^2\Bigl(x-\frac{z}{n}\Bigr)dz -\left(\int_\Real \phi(z) u_x\Bigl(x-\frac{z}{n}\Bigr) dz\right)^2\\ \nn 
& \geq \int_\Real \phi(z)u_x^2\Bigl(x-\frac{z}{n}\Bigr)dz-\left(\int_\Real \phi(z)dz\right)\left(\int_\Real \phi(z)u_x^2\Bigl(x-\frac{z}{n}\Bigr)dz\right)=0.
\end{align*}
As a consequence, as $\mu$ is a positive Radon measure and $\mu_{\rm
  ac}=u_x^2dx$, we see that
\begin{align*}
  \hat{\mu}_n(x)
  &= \int_\Real n\phi(n(x-y))d\mu(y)
  \ge \int_\Real n\phi(n(x-y))d\mu_{\rm ac}(y)\\
 & = \int_\Real n\phi(n(x-y)) u_x^2(x) dx
  \ge \Bigl(\int_\Real n\phi(n(x-y)) u_x(x) dx\Bigr)^2,
\end{align*}
and we may define $\hat{\rho}_n$ to be the non-negative root of
\begin{equation}\label{def:helpmeasure}
 \hat\rho_n^2(x)=\hat\mu_n(x)-u_{n,x}^2(x).
\end{equation}
Note that by construction $\hat{\rho}_n^2 \in L^1(\Real)$ and
$\hat{\rho}_n^2 \in C^\infty(\Real)$. The function $\hat{\rho}_n$ itself
need not be smooth, though.

{\bf Step 3. Smooth, approximating sequences $\rho_n$ and $\mu_n$.}

Let $\rho_n$ be defined by \eqref{1:eqB}, then 
\begin{equation}\label{2:eq}
\rho_n^2(x)=\hat\rho_n^2(x)+\frac{1}{n^{2}}.
\end{equation}
In particular, $\rho_n$ is well-defined, since the term within the square root is always positive.
Furthermore, we can decompose $\rho_n$ as
\[
\rho_n(x) = \frac{1}{n} + \bar{\rho}_n(x).
\]
Then
\[
\bar{\rho}_n(x) = -\frac{1}{n} +
\Bigl(\frac{1}{n^2}+\hat{\rho}_n^2(x)\Bigr)^{1/2} \ge 0,
\]
where we always take the positive root on the right-hand side.
Since $\hat{\rho}_n^2$ is smooth and the term within the square root is bounded
away from zero, it follows that $\bar{\rho}_n \in C^\infty(\Real)$
and consequently also $\mu_n\in C^\infty(\Real)$.
Note also that this implies that
\begin{equation}\label{positivity}
 \rho_n(x)\geq \frac{1}{n}\quad  \text{for all }x\in\Real. 
\end{equation}
Moreover, we have that
\begin{equation}\label{eq:quad}
\bar{\rho}_n^2(x)+\frac{2}{n}\bar{\rho}_n(x)=\hat\rho_n^2(x),
\end{equation}
which in particular implies that
\begin{equation}\label{size:rel}
\bar{\rho}_n^2(x) \le \hat\rho_n^2(x) \quad \text{ for all } x\in \Real.
\end{equation}

Next, we see, using the definition of $\mu_n$ in~\eqref{1:eqC} and the
equations~\eqref{size:rel} and~\eqref{def:helpmeasure}, that
\begin{equation}\label{def:measure}
\mu_n(x)=u_{n,x}^2(x)+\bar{\rho}_n^2(x)\le u_{n,x}^2(x)+\hat{\rho}_n^2(x)=\hat\mu_n(x)
\end{equation}
for all $x\in \Real$. As a consequence,
\begin{equation}\label{eq:barrhoL2}
\norm{u_{n,x}}_{L^2}^2 + \norm{\bar{\rho}_n}_{L^2}^2 = \norm{\mu_n}\leq
\norm{\hat\mu_n} =\norm{\mu},
\end{equation}
which in particular shows that $\mu_n$ is a finite Radon measure, but
also that $\bar{\rho}_n \in L^2(\Real)$ and therefore $\rho_n \in
L^2_{\rm const}(\Real)$.

So far, we have shown that $(u_n,\rho_n,\mu_n)$ is a sequence of
smooth functions contained in $\D$, and that $u_n \to u$ in
$H^1(\Real)$. 
It now remains to show that $F_n(x)\to F(x)$ at every point $x$ at
which $F$ is continuous, which is in this case equivalent to
$\mu_n\to \mu$ weakly, cf.~\cite[Props.~7.19 and
8.17]{Folland}. This means we have to prove that 
\begin{equation}\label{eq:vague}
\int_\Real \psi(x)d\mu_n(x)\to \int_\Real \psi(x)d\mu \quad \text{ as } n\to\infty
\end{equation}
for all $\psi\in C_c^\infty(\Real)$.
To that end observe first that, due to \eqref{1:eqC}, \eqref{def:helpmeasure}, and \eqref{eq:quad}, we have
\begin{equation}\label{eq:vaguevague}
\int_\Real\psi(x)d\mu_n(x)dx=\int_\Real \psi(x)d\hat\mu_n(x)-\frac2n\int_\Real\psi(x)\bar \rho_n(x)dx.
\end{equation}
We already know that $\hat\mu_n\to\mu$ weakly, that is, 
\begin{equation}\label{eq:vaguevaguevague}
\int_\Real \psi(x)d\hat\mu_n(x)\to \int_\Real \psi(x)d\mu(x) \quad \text{ as } n\to \infty,
\end{equation}
for all $\psi\in C_c^\infty(\Real)$. 
Moreover we obtain from~\eqref{eq:barrhoL2} and the Cauchy--Schwarz
inequality that
\[
\abs{\frac2n\int_\Real\psi(x)\bar \rho_n(x)dx} \le
\frac{2}{n}\norm{\psi}_{L^2}\norm{\bar{\rho}_n}_{L^2}
\le \frac{2}{n}\norm{\psi}_{L^2} \norm{\mu}^{1/2} \to 0
\]
for all $\psi\in C_c^\infty(\Real)$, which concludes the proof.
\end{proof}

\begin{remark} 
Note that one can show that the function $\frac1n \bar \rho_n$ converges pointwise to $0$. Indeed, according to \eqref{def:helpmeasure} and \eqref{eq:quad}, we have 
\begin{equation}\label{eq:helpquad}
0\leq \bar\rho_n^2(x)\leq \hat\rho_n^2(x)=\hat \mu_n(x)-u_{n,x}^2(x)\leq \hat\mu_n(x). 
\end{equation}
Moreover, from \eqref{def:hatmun} we get
\begin{align}\label{eq:helplimit}
 \frac{1}{n^2} \hat{\mu}_n(x)& = \frac{1}{n}\int_{-1}^0 \phi'(z)\Bigl(F\Bigl(x-\frac{z}{n}\Bigr)-F\Bigl(x+\frac{z}{n}\Bigr)\Bigr)dz
  \leq\frac{1}{n}\norm{\mu}\phi(0).
\end{align}
Thus combining \eqref{eq:helpquad} and \eqref{eq:helplimit} yields that the sequence $\frac1n \bar\rho_n$ is uniformly bounded and 
that 
\begin{equation*}
\frac1n \bar\rho_n\to 0 \quad \text{pointwise as } n\to \infty.
\end{equation*}
\end{remark}

\begin{remark}\label{remark:1}
In the next section we are not only going to use the splitting of $\hat\mu_n(x)$ into $u_{n,x}^2(x)$ and $\hat\rho_n^2(x)$ as introduced in \eqref{def:helpmeasure}, but also a second one, which we are introducing next. Namely, let $ F_{\rm s}(x)=\mu_{\rm s}((-\infty,x])$, where $\mu_{\rm s}$ denotes the  singular part of the measure $\mu$, and let $\phi$ be the Friedrichs mollifier. Define
\begin{equation}\label{eq:101}
 \tilde u_{n,x}^2(x)=\int_\Real \phi(z)u_x^2\Bigl(x-\frac{z}{n}\Bigr)dz
\end{equation}
and 
\begin{equation}\label{eq:102}
 \tilde \rho_n^2(x)=\int_\Real n\phi'(z) F_{\rm s}\Bigl(x-\frac{z}{n}\Bigr)dz.
\end{equation}
Then
\begin{equation}
\begin{aligned}\label{eq:103}
\tilde u_{n,x}^2(x)+\tilde\rho_n^2(x)& =\int_\Real n\phi'(z) F_{\rm s}\Bigl(x-\frac{z}{n}\Bigr)dz +\int_\Real \phi(z)u_x^2\Bigl(x-\frac{z}{n}\Bigr)dz\\
& =\int_\Real n\phi'(z) F\Bigl(x-\frac{z}{n}\Bigr)dz= \hat\mu_n(x).
\end{aligned}
\end{equation}
\end{remark}

\begin{remark}
Let $(u,0,\mu)\in\D$, $n\in\mathbb{N}$, and $(u_n,\rho_n,\mu_n)\in\D$ be defined as in Theorem~\ref{thm:approxEuler}.  By construction we then have that $u_n$, $\rho_n\in C^\infty(\Real)$, $\mu_n$ is absolutely continuous, and, according to \eqref{positivity},  that 
$\rho_n(x)\geq \frac{1}{n}$ for all $x\in\Real$. Hence \cite[Cor.~6.2]{GHR4} implies that the corresponding solution $(u_n(t), \rho_n(t),\mu_n(t))$ has the same regularity for all times $t$, and, in particular, no wave breaking occurs.
\end{remark}

\section{Convergence in Lagrangian coordinates} \label{sec:3}

The aim of this section is to show that the smooth approximating sequence constructed in Theorem~\ref{thm:approxEuler} not only converges in the set of Eulerian coordinates $\D$ but also in the set of Lagrangian coordinates $\F$. Hence, we are first going to introduce the set of Lagrangian coordinates $\F$ and the mapping $L$ from $\D$ to $\F$, before stating and proving the outlined convergence theorem. 

Let $V$ be the Banach space defined by 
\begin{equation*}
V=\{f\in C_b(\Real) \ |\ f_\xi\in\Ltwo\},
\end{equation*}
where $C_b(\Real)=C(\Real)\cap\Linf$ and the norm
is given by
$\norm{f}_V=\norm{f}_{L^\infty}+\norm{f_\xi}_{L^2}$. Let moreover
\begin{equation*}
  E=V\times H^1(\Real)\times L^2(\Real)\times L^2(\Real)\times \Real, 
\end{equation*}
then $E$ equipped with the norm 
\begin{equation}\label{norm:E}
\norm{(\zeta, U, h, \bar r, k)}_E=\norm{\zeta}_V+\norm{U}_{H^1}+\norm{h}_{L^2}+\norm{\bar r}_{L^2}+\vert k\vert
\end{equation}
is a Banach space. Note that we can associate to each $(\zeta, U, h, \bar r, k)\in E$ the tuple $(y,U,h,r)$ by setting 
\begin{equation}\label{BN}
y= \zeta +\id\quad \text{ and } \quad r=\bar r+ky_\xi.
\end{equation} 
 Conversely, for any pair $(y,r)$ such that $y-\id \in V$ and $r\in L^2_{\rm const}(\Real)$ there exists a unique triplet $(\zeta, \bar r, k)\in [L^2(\Real)]^2\times\Real$ such that \eqref{BN} holds.
For more details we refer to \cite[Sect.~ 3]{GHR4}. In what follows we will slightly abuse the notation by writing $(y,U,h,r) \in E$ instead of $(\zeta, U, h, \bar r, k)\in E$.

In addition we have to introduce the set of relabeling functions, which are not only needed for identifying equivalence classes in Lagrangian coordinates, but also for determining the set of Lagrangian coordinates.
\begin{definition}[Relabeling functions]\label{def:rf}
 We denote by $G$ the subgroup of the group of homeomorphisms $f$ of $\Real$ such that 
\begin{subequations}
\label{eq:Gcond}
 \begin{align}
  \label{eq:Gcond1}
  f-\id \text{ and } f^{-1}-\id &\text{ both belong to } W^{1,\infty}(\Real), \\
  \label{eq:Gcond2}
  f_\xi-1 &\text{ belongs to } L^2(\Real),
 \end{align}
\end{subequations}
where $\id$ denotes the identity function. 

Given $\kappa\ge0$, we denote by $G_\kappa$ the subset of $G$ defined by 
\begin{equation*}
 G_\kappa=\{ f\in G\mid  \norm{f-\id}_{W^{1,\infty}}+\norm{f^{-1}-\id}_{W^{1,\infty}}\leq\kappa\}. 
\end{equation*}
\end{definition}

We are now ready to introduce the set of Lagrangian coordinates of the 2CH system (and hence also of the CH equation). The case of the CH equation corresponds to $r(\xi)=0$ for all $\xi\in\Real$.
\begin{definition}[Lagrangian coordinates]\label{def:F} The set $\F$ is composed of all tuples $X=(y,U,h,r)\in E$, such that 
 \begin{subequations}
\label{eq:lagcoord}
\begin{align}
\label{eq:lagcoord1}
& (\zeta, U,h,r)\in [W^{1,\infty}(\Real)]^2\times [L^\infty(\Real)]^2,\\
\label{eq:lagcoord2}
&y_\xi\geq0, \quad h\geq0, \quad y_\xi+h>0
\text{  almost everywhere},\\
\label{eq:lagcoord3}
&y_\xi h=U_\xi^2+\bar r^2\text{ almost everywhere},\\
\label{eq:lagcoord4}
&y+H\in G,
\end{align}
\end{subequations}
where we denote $y(\xi)=\zeta(\xi)+\xi$ and $H(\xi)=\int_{-\infty}^\xi h(\eta)d\eta$.
\end{definition}

Moreover, we set 
\begin{equation*}
\F_\kappa=\{X\in\F \mid y+H\in G_\kappa\}.
\end{equation*}
Observe that
\begin{equation}\label{F0}
\F_0=\{X\in \F \mid y(\xi)+H(\xi)=\xi \text{ for all }\xi\in\Real\}.
\end{equation}

We note that the group $G$ acts on $\F$ by means of right composition of the form
\[
X=(y,U,h,r)  \mapsto X \circ g := (y\circ g, U\circ g, (h\circ g)g_\xi,(r\circ g)g_\xi).
\]
This group action then allows us to define equivalence classes of
Lagrangian coordinates, where we say that two coordinates $X$
and $\hat{X}$ are equivalent, if there exists some $g\in G$ such that $\hat{X} = X \circ g$.

Given an arbitrary $X=(y,U,h,r)$, we note that $y+H\in G$ and hence also $(y+H)^{-1} \in G$. 
In particular, if we introduce 
\begin{equation}\label{eq:Gammadef}
\Gamma(X)=X\circ (y+H)^{-1},
\end{equation}
then a short computation yields that $\Gamma(X)\in \F_0$.
This shows that every equivalence class $X \circ G$ of Lagrangian coordinates
has a unique canonical representative $\Gamma(X)$ in $\F_0$.
Moreover, it has been shown in \cite[Lem.~4.6]{GHR4}
that the mapping $\Gamma|_{\F_\kappa}\colon \F_\kappa \to \F_0$
is continuous for each $\kappa > 0$.

Finally we can introduce the mapping $L$ from Eulerian to Lagrangian coordinates.
\begin{theorem}[{\cite[Thm.~4.9]{GHR4}}]
\label{th:Ldef}
For any $(u,\rho,\mu)$ in $\D$, let
\begin{subequations}
\label{eq:Ldef}
\begin{align}
\label{eq:Ldef1}
y(\xi)&=\sup\left\{y\ |\ \mu((-\infty,y))+y<\xi\right\},\\
\label{eq:Ldef2}
h(\xi)&=1-y_\xi(\xi),\\
\label{eq:Ldef3}
U(\xi)&=u\circ{y(\xi)},\\
\label{eq:Ldef4}
r(\xi)&=\rho\circ{y(\xi)}y_\xi(\xi).
\end{align}
\end{subequations}
Then $(y,U,h,r)\in\F_0$. We denote by $L\colon \D\rightarrow \F_0$ the mapping which to any element $(u,\rho,\mu)\in\D$ associates $X=(y,U,h,r)\in \F_0$ given by \eqref{eq:Ldef}.  
\end{theorem}

In the case of the CH equation, we have $r(\xi)=0$ for all $\xi\in\Real$.

\begin{theorem}\label{thm:approxLagran}
Let $(u,0,\mu)\in \D$, and let $(u_n,\rho_n,\mu_n)\in\D$ be the corresponding approximating sequence defined in Theorem~\ref{thm:approxEuler}. Moreover, let $(y,U,h,0)=L((u,0,\mu))$ 
and $(y_n,U_n,h_n,r_n)=L((u_n,\rho_n,\mu_n))$.  Then
\begin{equation*}
(y_n,U_n,h_n,r_n)\to (y,U,h,0)\quad \text{in } E.
\end{equation*}
\end{theorem}

\begin{proof}

Let $(u,0,\mu)\in \D$ and $(u_n,\rho_n,\mu_n)\in\D$ be the approximating series defined in Theorem~\ref{thm:approxEuler}. Furthermore, let $X=(y,U,h,0)=L((u,0,\mu))$ and $X_n=(y_n,U_n,h_n,r_n)=L((u_n,\rho_n,\mu_n))$, which yields a smooth sequence in Lagrangian coordinates, cf.~\cite[Proof of Thm.~6.1]{GHR4}. 
However, due to the construction of our approximating sequence $(u_n,\rho_n,\mu_n)$, it turns out that in order to prove that $X_n\to X$ in Lagrangian coordinates, it is better to introduce another sequence $\hat X_n=(\hat u_n, \hat \rho_n, \hat \mu_n)$ which is linked to the sequence $X_n$ via relabeling. For better understanding, we split the proof into several steps. After first defining the new sequence $\hat X_n$, we show that for every $n\in \mathbb{N}$ there exists $g_n\in\Gr$ such that $\hat X_n=X_n\circ g_n$ (Step 1). Thereafter, we establish that $\hat X_n\to X $ in $E$ (Steps 2--9). Finally, we  show that $\hat X_n\to X$ implies $X_n\to X$ in $E$ (Step 10). The situation is also depicted in Figure~\ref{Figure}.

\begin{figure}
\begin{displaymath}
\xymatrix{(u_n,\rho_n,\mu_n)\to (u,0,\mu) \ar[r]_{\text{Step 1}} \ar[d]_{\text{Theorem}~\ref{thm:approxLagran}}& \hat X_n=X_n\circ g_n \ar[d] _{\text{Steps 2--9}}\\
(y_n,U_n,h_n,r_n)\to (y,U,h,r)& (\hat y_n, \hat U_n, \hat h_n, \hat r_n)\to (y,U,h,r) \ar[l]_{\text{Step 10}}} 
\end{displaymath}
\caption{Outline of the proof of Theorem~\ref{thm:approxLagran}.}
\label{Figure}
\end{figure}

{\bf Step 1. Definition of the sequence $\hat X_n$  and proof that $\hat X_n=X_n\circ g_n$.} 
Define\footnote{This construction resembles the one used in Step 2 of the proof of Theorem \ref{thm:approxEuler}. However, here we perform the construction in Lagrangian variables.} 
$\hat F_{n}(x)$ by
\begin{equation}\label{def:hatFn1}
\hat F_{n}(x)= \int_\Real n\phi(n(x-y))F(y)dy=\hat\mu_{n}((-\infty, x]),
\end{equation}
such that 
\begin{equation*}
\hat\mu_{n}(x)=\int_{-1}^0n\phi'(z)\Bigl(F\Bigl(x-\frac{z}{n}\Bigr)-F\Bigl(x+\frac{z}{n}\Bigr)\Bigr)dz.
\end{equation*}
Introduce $\hat\rho_n^2=\hat\mu_{n}-u_{n,x}^2$. Then $\hat \mu_n=u_{n,x}^2+\bar\rho_n^2+\frac2n\bar \rho_n=u_{n,x}^2+\hat\rho_n^2$, 
as in \eqref{def:helpmeasure}--\eqref{eq:quad}.
Let now $\hat X_n=(\hat y_n, \hat U_n, \hat h_n,\hat r_n)\in \F$, where 
\begin{subequations}
\label{eq:Ldefneu}
\begin{align}
\label{eq:Ldef1neu}
\hat y_n(\xi)&=\sup\left\{y\ |\ \hat\mu_n((-\infty,y))+y<\xi\right\}, \\
\label{eq:Ldef2neu}
\hat h_n(\xi)&=1-\frac2n \bar{\hat{r}}_n(\xi)-\hat y_{n,\xi}(\xi),\\
\label{eq:Ldef3neu}
\hat U_n(\xi)&=u_n\circ{\hat y_n(\xi)},\\
\label{eq:Ldef6neu}
\bar{\hat{r}}_n(\xi)&=\bar\rho_n\circ{\hat y_n(\xi)}\hat y_{n,\xi}(\xi),\\
\label{eq:Ldef4neu}
\hat r_n(\xi)&=\rho_n\circ{\hat y_n(\xi)}\hat y_{n,\xi}(\xi).
\end{align}
\end{subequations}

We are going to show that we can write $\hat X_n=X_n\circ g_n$ for some $g_n\in \Gr$,
that is,\footnote{Note the factors $g_{n,\xi}(\xi)$.}
\begin{align}\label{eq:circ}
\hat X_n(\xi)&=(\hat y_n(\xi), \hat U_n(\xi), \hat h_n(\xi), \hat r_n(\xi)) \\
&=\big(y_n(g_n(\xi)),U_n(g_n(\xi)), h_n(g_n(\xi))g_{n,\xi}(\xi),  r_n(g_n(\xi))g_{n,\xi}(\xi)\big) \notag \\
&=X_n\circ g_n(\xi),\notag
\end{align}
which implies immediately that $\hat X_n\in \F$ and that it belongs to the same equivalence class as $X_n$.
Additionally, we will show that there exists some $\kappa$ independent of $n$
such that $g_n \in \Gr_\kappa$ for all $n \in \mathbb{N}$.

Since both $\mu_n$ and $\hat\mu_n$ are smooth and purely absolutely continuous, we have that 
\begin{equation}\label{eq:2b}
y_n(\xi)+F_n(y_n(\xi))=\xi,
\end{equation}
and
\begin{equation}\label{eq:2}
\hat y_n(\xi)+\hat F_n(\hat y_n(\xi))=\xi,
\end{equation}
for all $\xi\in\Real$. Moreover, recall that 
\begin{equation*}
\mu_n(x)+\frac2n \bar\rho_n(x)=\hat\mu_n(x) \quad \text{for all }x\in\Real
\end{equation*}
according to \eqref{eq:quad}, \eqref{def:helpmeasure} and \eqref{1:eqC}, 
and $\frac2n \bar\rho_n\in L^1(\Real)$.
Hence we can rewrite \eqref{eq:2} as 
\begin{equation}\label{def:gn}
\hat y_n(\xi)+\hat H_n(\xi)=\hat y_n(\xi)+F_n(\hat y_n(\xi))=\xi-\frac 2n \int_{-\infty}^{\hat y_n(\xi)} \bar \rho_n(x)dx=g_n(\xi),
\end{equation}
which defines $g_n(\xi)$. Here $\hat H_n(\xi)=\int_{-\infty}^\xi \hat h_n(\eta)d\eta$. Moreover, using \eqref{eq:2b} we have
\begin{equation*}
y_n(g_n(\xi))+F_n(y_n(g_n(\xi)))=g_n(\xi) \quad \text{ for all } \xi\in\Real,
\end{equation*}
and, since $\id+F_n$ is strictly increasing,  we conclude that 
\begin{equation} \label{eq:rel}
\hat y_n(\xi)=y_n(g_n(\xi)) \quad \text{ for all } \xi\in\Real,
\end{equation}
which immediately implies that 
\[
\hat{U}_n(\xi) = u_n(\hat{y}_n(\xi)) = u_n(y_n(g_n(\xi))) = U_n(g_n(\xi))
\quad \text{ for all } \xi \in \Real.
\]
Using \eqref{eq:Ldef2neu}, \eqref{eq:Ldef6neu}, \eqref{def:gn}, \eqref{eq:rel}, and \eqref{eq:Ldef4} we infer that
\begin{align}
\hat h_n(\xi)&=1-\frac2n \bar{\hat{r}}_n(\xi)-\hat y_{n,\xi}(\xi)
=g_{n,\xi}(\xi)-\hat y_{n,\xi}(\xi) \notag\\
&=\bigl(1-y_{n,\xi}(g_n(\xi))\bigr)g_{n,\xi}(\xi) 
=h_n(g_n(\xi))g_{n,\xi}(\xi).  \label{eq:h_n}
\end{align}
In addition, we see that
\begin{equation*}
\hat r_n(\xi)=\rho_n(\hat y_n(\xi))\hat y_{n,\xi}(\xi)= r_n(g_n(\xi))g_{n,\xi}(\xi).
\end{equation*}
Thus we conclude that  $\hat X_n=X_n\circ g_n$, and it remains to show that $g_n\in \Gr_\kappa$
for some $\kappa$ independent of $n$.

Instead of checking that $g_n$ satisfies all the properties listed in Definition~\ref{def:rf}, we are going to apply \cite[Lem.~3.2]{HR}. Namely, if $g_n$ is absolutely continuous, $g_{n,\xi}-1\in L^2(\Real)$, and there exist $c_1\geq 1$ and $c_2 > 0$ such that $\frac{1}{c_1}\leq g_{n,\xi}(\xi)\leq c_1$ almost everywhere and $\norm{g_n-\id}_{L^\infty}\le c_2$,
then $g_n\in G_{\kappa}$ for some $\kappa>0$ depending only on $c_1$ and $c_2$. 
By construction, $\hat X_n$ is smooth and therefore $g_n$ is smooth and, in particular, absolutely continuous. Since $\frac2n \bar\rho_n\in L^1(\Real)$ and $\bar\rho_n$ is strictly positive, we get from \eqref{def:gn} that $g_n-\id\in L^\infty(\Real)$, and from \eqref{def:helpmeasure}, \eqref{eq:quad}, and \eqref{eq:barrhoL2} that $\norm{g_n-\id}_{L^\infty}\leq \norm{\mu}$. Moreover,
using the notation in the proof of Theorem~\ref{thm:approxEuler}, we obtain from \eqref{eq:2} and~\eqref{def:helpmeasure} that
\begin{equation}\label{est:abl}
\hat y_{n,\xi}(\xi)=\frac1{1+u_{n,x}^2(\hat y_n(\xi))+\hat\rho_n^2(\hat y_n(\xi))}\leq 1.
\end{equation}
Thus
\begin{align} \nn 
\frac12&\leq \frac{1+u_{n,x}^2(\hat y_n(\xi))+\bar\rho_n^2(\hat y_n(\xi))}{1+\frac1{n^2}+u_{n,x}^2(\hat y_n(\xi))+2\bar\rho_n^2(\hat y_n(\xi))} \\ \nn
&\leq\frac{1+u_{n,x}^2(\hat y_n(\xi))+\bar\rho_n^2(\hat y_n(\xi))}{1+u_{n,x}^2(\hat y_n(\xi))+\bar \rho_n^2(\hat y_n(\xi))+\frac{2}{n}\bar\rho_n(\hat y_n(\xi))} \\
\label{est:Gkappa}
& = 1-\frac{\frac2n\bar\rho_n(\hat y_n(\xi))}{1+u_{n,x}^2(\hat y_n(\xi))+\hat\rho_n^2(\hat y_n(\xi))} \nn \\
&= g_{n,\xi}(\xi)\leq 1,
\end{align}
due to \eqref{eq:quad} and \eqref{def:gn}. 
Finally, we have to check that $g_{n,\xi}-1\in L^2(\Real)$. Direct computations, using \eqref{def:gn} and \eqref{est:abl}, yield
\begin{equation*}
\int_\Real (g_{n,\xi}(\xi)-1)^2d\xi=\frac4{n^2}\int_\Real \bar\rho_n^2(\hat y_n(\xi))\hat y_{n,\xi}^2(\xi)d\xi
 \leq \frac4{n^2}\int_\Real \bar \rho_n^2(x)dx.
\end{equation*}
Thus $g_{n,\xi}-1\in L^2(\Real)$, since $\bar \rho_n\in L^2(\Real)$. 
Thus \cite[Lem.~3.2]{HR} implies that $g_n$ is a relabeling fuction
and that there exists $\kappa > 0$ independent of $n$ such that $g_n \in \Gr_\kappa$ for all $n \in \mathbb{N}$.

{\bf Step 2: The sequence $\hat y_n-\id$ converges to $y-\id$ in $L^\infty(\Real)$.}  Recall that we have by definition that 
\begin{equation}\label{3.22}
y(\xi)+F(y(\xi)-)\leq \xi\leq y(\xi)+F(y(\xi)) \quad \text{ for all }\xi\in\Real,
\end{equation}
where $F(x)=\mu((-\infty,x])$.
Moreover, since $\hat \mu_n(x)$ is smooth and purely absolutely continuous, we have
\begin{equation}\label{3.23}
\hat y_n(\xi)+\hat F_n(\hat y_n(\xi))=\xi \quad \text{ for all } \xi\in\Real,
\end{equation}
where $\hat F_n(x)=\hat \mu_n((-\infty,x])$.
Introducing
\begin{equation}\label{3.24}
G(x):=x+F(x) \quad \text{ and } \quad \hat G_n(x):=x+\hat F_n(x),
\end{equation}
we conclude that
\begin{align*}
\hat G_n(\hat y_n(\xi))& =\hat y_n(\xi)+ \hat F_n(\hat y_n(\xi))\\ 
& = \int_{-1}^1 \phi(z)\Bigl(\hat y_n(\xi)-\frac{z}{n}+F\Bigl(\hat y_n(\xi)-\frac{z}{n}\Bigr)\Bigr)dz \\ 
& =\int_{-1}^1 \phi(z)G\Bigl(\hat y_n(\xi)-\frac{z}{n}\Bigr)dz,
\end{align*}
where we used \eqref{def:hatFn1}.
Moreover, since $G(x)$ is strictly increasing and due to \eqref{3.23} and \eqref{3.24}, one has that
\begin{equation}\label{3.28}
G\Bigl(\hat y_n(\xi)-\frac{1}{n}\Bigr)<\hat G_n (\hat y_n(\xi))=\xi< G\Bigl(\hat y_n(\xi)+\frac{1}{n}-\Bigr),
\end{equation}
and by \eqref{3.22} that 
\begin{equation}\label{3.29}
G(y(\xi)-)\leq\xi\leq G(y(\xi)).
\end{equation}
Combining \eqref{3.28} and \eqref{3.29} yields on the one hand that 
\begin{equation*}
G\Bigl(\hat y_n(\xi)-\frac1n\Bigr)< \xi\leq G(y(\xi))
\end{equation*}
and on the other hand that 
\begin{equation*}
G(y(\xi)-)\leq \xi < G\Bigl(\hat y_n(\xi)+\frac1n-\Bigr).
\end{equation*}
Recalling that $G(x)$, and hence also $G(x-)$, is strictly increasing, we obtain that 
\begin{equation*}
\hat y_n(\xi)-\frac{1}{n}<y(\xi)<\hat y_n(\xi)+\frac1n
\end{equation*}
or, equivalently,
\begin{equation*}
-\frac1n <y(\xi)-\hat y_n(\xi)<\frac1n \quad \text{ for all } \xi\in\Real.
\end{equation*}
In particular, this shows that $\norm{\hat y_n-y}_{L^\infty} \to 0$ as $n\to \infty$.

{\bf Step 3. Convergence of $\tilde h_n$ to $h$ in $L^1(\Real)$.}  
Let 
\begin{equation}\label{eq:Ldef5neu}
\tilde h_n=1-\hat y_{n,\xi}.
\end{equation}
To show that $\tilde h_n\to h$ in $L^1(\Real)$ is the main (and most difficult) step. 

Due to our change from Eulerian to Lagrangian coordinates, it is not clear at first sight that $\tilde h_n$ and $h$ belong to $L^1(\Real)$. We know that $hy_\xi=U_\xi^2$, or,  equivalently, $h=U_\xi^2+h^2$, because $y_\xi+h=1$. However, since $U_\xi$ and $h$ both belong to $L^2(\Real)$, it follows that $h\in L^1(\Real)$. Combining \eqref{eq:Ldef5neu} and \eqref{eq:Ldef2neu},  and recalling that $\hat X\in \F$ which implies that \eqref{eq:lagcoord3} is satisfied, one obtains
\begin{equation}\label{eq:tildehnyn}
\tilde h_n=
\tilde h_n\hat y_{n,\xi}+\tilde h_n^2= \hat h_n\hat y_{n,\xi}+\frac2n \bar{\hat{r}}_n\hat y_{n,\xi}+\tilde h_n^2= \hat U_{n,\xi}^2+\bar{\hat{r}}_n^2+\frac2n \bar{\hat{r}}_n\hat y_{n,\xi}+\tilde h_n^2,
\end{equation}
where $\bar{\hat{r}}_n$, $\tilde h_n$, $\hat U_{n,\xi}\in L^2(\Real)$, $\frac2n \bar{\hat{r}}_n\in L^1(\Real)$ and $\hat{y}_{n,\xi} \in L^\infty(\Real)$. Thus also $\tilde h_n\in L^1(\Real)$.
 Define 
\begin{equation}\label{eq:7}
 \tilde H_n(\xi)=\int_{-\infty}^\xi \tilde h_n(\eta)d\eta \quad \text{and}\quad H(\xi)=\int_{-\infty}^{\xi}h(\eta)d\eta. 
\end{equation}
Then the identities $\tilde H_n+\hat y_n=\id$ and $H+y=\id$ together with the pointwise convergence of $\hat y_n\to y$ imply that $\tilde H_n$ converges pointwise to $H$ and $\tilde H_n(\infty)=\norm{\hat\mu_n}_{L^1}$ and $H(\infty)=\norm{\mu}$. In particular, this means that
$\tilde H_n(\infty)= H(\infty)$ for all $n\in\Natural$ and hence 
\begin{equation}\label{eq:tildehnynNorm}
\bnorm{\tilde h_n}_{L^1}= \norm{h}_{L^1}, \quad n\in\Natural.
\end{equation}

Next we will prove that $\tilde h_n$ converges to $h$ pointwise almost everywhere, which will imply that $\tilde h_n\to h$ in $L^1(\Real)$, see \cite[Prop.~1.33]{Ambrosio}. To that end, observe first that $\tilde h_n(\xi)-h(\xi)=y_\xi(\xi)-\hat y_{n,\xi}(\xi)$ for all $\xi\in\Real$. Thus it suffices to show that $\hat y_{n,\xi}$ converges pointwise to $y_\xi$ almost everywhere. Recalling \eqref{3.23}, \eqref{3.24}, and that $\hat y_n$ is smooth, we see that this is equivalent to showing that 
\begin{equation}\label{goal2}
\hat G_n'(\hat y_n(\xi))\to \frac{1}{y_\xi(\xi)} \quad \text{ for almost every } \xi\in\Real.
\end{equation}
Moreover, note that 
\begin{align}\nonumber
\hat G_n'(\hat y_n(\xi))& =n \int_{-1}^1 \phi'(z)G\Bigl(\hat y_n(\xi)-\frac{z}{n}\Bigr)dz \\ \nonumber
& = n\int_{-1}^1 \phi'(z) \Bigl(G\Bigl(\hat y_n(\xi)-\frac{z}{n}\Bigr)-\xi\Bigr)dz\\ \label{def:Gpr}
& = n^2\int_{\hat y_n(\xi)-\frac1n}^{\hat y_n(\xi)+\frac1n} \phi'(n(\hat y_n(\xi)-z)) (G(z)-\xi)dz.
\end{align} 
Introducing the strictly increasing function $\tilde G(z)=G(z)-\xi$, it follows that we have to show that 
\begin{equation}\label{eq:Gnptoyxi}
\hat G_n'(\hat y_n(\xi))= n^2\int_{\hat y_n(\xi)-\frac1n}^{\hat y_n(\xi)+\frac1n}\phi'(n(\hat y_n(\xi)-z))\tilde G(z)dz\to \frac{1}{y_\xi(\xi)}
\end{equation}
for almost every $\xi\in\Real$.

In fact, we will show below that~\eqref{eq:Gnptoyxi} holds at every $\xi \in \Real$ where the function $y$ is differentiable. Since $y$ is Lipschitz continuous and therefore differentiable almost everywhere, this will prove the convergence of $\tilde{h}_n$ to $h$ in $L^1(\Real)$.
In the proof of~\eqref{eq:Gnptoyxi}, we will consider seperately the cases where
the derivative of $y$ is zero, and where it is strictly positive.

{\bf (a) The case $y_\xi(\xi)=0$}.  
We have to show that $\hat G_n'(\hat y_n(\xi))\to\infty$ as $n\to\infty$.
By assumption $y_\xi(\xi)=0$ and hence for every $\varepsilon>0$ there exists some $\delta_\varepsilon>0$ such that 
\begin{equation}\label{est:quot}
0\leq \frac{y(\eta)-y(\xi)}{\eta-\xi}<\varepsilon \quad \text{ whenever } \vert \xi-\eta\vert <\delta_\varepsilon.
\end{equation}
Define $\gamma_\varepsilon:= \varepsilon\delta_\varepsilon$ and let $z\in \Real$ such that $\vert z-y(\xi)\vert <\gamma_\varepsilon$. In addition, recall \eqref{3.22} and \eqref{3.24}, and observe that $y(G(z))=z$ for all $z\in \Real$. If $\vert G(z)-\xi\vert <\delta_\varepsilon$, we have by \eqref{est:quot} that
\begin{equation*}
\frac{G(z)-\xi}{z-y(\xi)}=\frac{G(z)-\xi}{y(G(z))-y(\xi)}>\frac1{\varepsilon}.
\end{equation*}
On the other hand, if $\vert G(z)-\xi\vert \geq \delta_\varepsilon$, then 
\begin{equation*}
\frac{G(z)-\xi}{z-y(\xi)}=\frac{\vert G(z)-\xi\vert }{\vert z-y(\xi)\vert} > \frac{\delta_\varepsilon}{\gamma_\varepsilon}=\frac1{\varepsilon}.
\end{equation*}
Thus
\begin{equation}\label{help:est}
\frac{G(z)-\xi}{z-y(\xi)}>\frac{1}{\varepsilon} \quad \text{ whenever } \vert z-y(\xi)\vert <\gamma_\varepsilon.
\end{equation}

In the remainder of this subsection we are going to show that there exists a constant $C>0$ independent of $n$ and $\varepsilon$ such that 
\begin{equation}\label{goal1}
\hat G_n'(\hat y_n(\xi))>\frac{C}{\varepsilon} \quad \text{ for all } n \text{ such that } \frac2n<\gamma_\varepsilon,
\end{equation}
which will prove the claim.
Let 
\begin{equation}\label{def:I0}
I_0:= \{ z\mid \vert z-\hat y_n(\xi)\vert \leq \vert y(\xi)-\hat y_n(\xi)\vert \}=\hat y_n(\xi)+\{ t \mid \vert t\vert \leq \vert y(\xi)-\hat y_n(\xi)\vert \}.
\end{equation}
Direct computations show that for all $z\in \Real\setminus I_0$
\begin{equation}\label{goal11}
\phi'(n(\hat y_n(\xi)-z))\tilde G(z)\geq 0
\end{equation}
and that 
\begin{align}\nn
\int_{I_0} \phi'(n(\hat y_n(\xi)-z)) &\tilde G(z)dz\\
&= \int_{-\vert \hat y_n(\xi)-y(\xi)\vert}^{\vert \hat y_n(\xi)-y(\xi)\vert } \phi'(nz)\tilde G(\hat y_n(\xi)-z) dz \nn\\ \label{goal12}
& = \int_0^{\vert \hat y_n(\xi)-y(\xi)\vert} \phi'(nz)\big(\tilde G(\hat y_n(\xi)-z)-\tilde G(\hat y_n(\xi)+z)\big)dz \geq 0,
\end{align}
since both terms in the last integral are non-positive on the interval of integration.

Again, we have to consider two situations seperately depending on
the difference of $y(\xi)$ and $\hat{y}_n(\xi)$.

{\it (a.I) The case $\vert \hat y_n(\xi)-y(\xi)\vert \leq \frac1{2n}$.} We only prove \eqref{goal1} in the case $y(\xi)\leq \hat y_n(\xi)\leq y(\xi)+\frac{1}{2n}$ and leave the other case, which follows the same lines, to the interested reader. Using \eqref{def:Gpr}, \eqref{help:est}, \eqref{goal11}, and \eqref{goal12} we have 
\begin{align*}
\hat G_n'(\hat y_n(\xi))
& =n^2 \int_{\hat y_n(\xi)-\frac1n}^{y(\xi)} \phi'(n(\hat y_n(\xi)-z))\tilde G(z)dz\\
& \quad + n^2 \int_{I_0} \phi'(n(\hat y_n(\xi)-z))\tilde G(z)dz\\
& \quad +n^2\int_{2\hat y_n(\xi)-y(\xi)}^{\hat y_n(\xi)+\frac1n} \phi'(n(\hat y_n(\xi)-z))\tilde G(z)dz\\
& \geq n^2\int_{\hat y_n(\xi)-\frac1n}^{y(\xi)} \phi'(n(\hat y_n(\xi)-z))\tilde G(z)dz\\
& \geq n^2\int_{\hat y_n(\xi)-\frac1n}^{y(\xi)} \phi'(n(\hat y_n(\xi)-z))\frac{z-y(\xi)}{\varepsilon}dz\\
& = \frac{n}{\varepsilon}\int_{\hat y_n(\xi)-\frac1n}^{y(\xi)}\phi(n(\hat y_n(\xi)-z))dz\\
& \geq \frac{n}{\varepsilon} \int_{\hat y_n(\xi)-\frac1n}^{\hat y_n(\xi)-\frac1{2n}} \phi(n(\hat y_n(\xi)-z))dz\\
& = \frac{1}{\varepsilon} \int_{\frac12}^1 \phi(z)dz= \frac{C}{\varepsilon}
\end{align*}
with $C = \int_{1/2}^1 \phi(z)dz$.
Here we applied \eqref{help:est} to $\tilde G(z)=G(z)-\xi$, which is satisfied since we assume that $\frac2n<\gamma_\varepsilon$. 

{\it (a.II) The case $\frac1{2n}<\vert \hat y_n(\xi)-y(\xi)\vert < \frac{1}{n}$.} We only prove \eqref{goal1} in the case $y(\xi)+\frac1{2n}< \hat y_n(\xi)< y(\xi)+\frac1n$ and leave the other case, which follows the same lines, to the interested reader. Due to \eqref{goal11} and \eqref{goal12} we have 
\begin{align*}
\hat G_n'(\hat y_n(\xi))& = n^2 \int_{\hat y_n(\xi)-\frac1n}^{y(\xi)} \phi'(n(\hat y_n(\xi)-z))\tilde G(z)dz\\
& \quad + n^2 \int_{I_0} \phi'(n(\hat y_n(\xi)-z))\tilde G(z)dz\\
& \quad + n^2\int_{2\hat y_n(\xi)-y(\xi)}^{\hat y_n(\xi)+\frac1n} \phi'(n(\hat y_n(\xi)-z))\tilde G(z)dz\\
& \geq n^2\int_{\hat y_n(\xi)-\frac1n}^{y(\xi)} \phi'(n(\hat y_n(\xi)-z))\tilde G(z)dz\\ 
& = n\int_{n(\hat y_n(\xi)-y(\xi))}^{1} \phi'(z)\tilde G\Bigl(\hat y_n(\xi)-\frac{z}{n}\Bigr)dz.
\end{align*}
Let us turn our attention to the last integral 
\begin{equation*} 
n\int_{n(\hat y_n(\xi)-y(\xi))}^1 \phi'(z)\tilde G\Bigl(\hat y_n(\xi)-\frac{z}{n}\Bigr)dz,
\end{equation*}
where $n(\hat y_n(\xi)-y(\xi))\in (\frac12, 1)$. Since $\tilde G(\hat y_n(\xi)-\frac{z}{n})$ is strictly decreasing and $\tilde G(\hat y_n(\xi)-\frac{z}{n}-)\leq 0$ for all $z\in [n(\hat y_n(\xi)-y(\xi)),1 ]$, we have 
\begin{align*}
0& \leq n\int_{n(\hat y_n(\xi)-y(\xi))}^1 \phi'(z) \tilde G(\hat y_n(\xi)-\frac{z}{n})dz \\
& =-n\int_{n(\hat y_n(\xi)-y(\xi))}^1 \int_{\tilde G(\hat y_n(\xi)-\frac{z}{n})}^0\phi'(z)dp\,dz.
\end{align*}
Since the area of integration has finite measure and the integrand is uniformly bounded, we can interchange the order of integration and get
\begin{align*}
-n\int_{n(\hat y_n(\xi)-y(\xi))}^1& \int_{\tilde G(\hat y_n(\xi)-\frac{z}{n})}^0 \phi'(z)dp\,dz\\
& = -n\int_{\tilde G(\hat y_n(\xi)-\frac{1}{n})} ^0\int_{\max(n(\hat y_n(\xi)-y(\xi)),n(\hat y_n(\xi)- \tilde G^{-1}(p)))}^{1}\phi'(z)dz\,dp.
\end{align*}
Evaluating the inner integral and using that $\phi(z)$ is decreasing on $[0,1]$, we end up with 
\begin{align*}
-n&\int_{\tilde G(\hat y_n(\xi)-\frac{1}{n})} ^0 \int_{\max(n(\hat y_n(\xi)-y(\xi)),n(\hat y_n(\xi)- \tilde G^{-1}(p)))}^{1}\phi'(z)dz\,dp\\
& = n \int_{\tilde G(\hat y_n(\xi)-\frac{1}{n})}^0 \phi\bigl(\max (n(\hat y_n(\xi)-y(\xi)), n(\hat y_n(\xi)-\tilde G^{-1}(p)))\bigr) dp\\
& \geq n\int_{\tilde G(\hat y_n(\xi)-\frac1n)}^0 \int_{\max(n(\hat y_n(\xi)-y(\xi)), n(\hat y_n(\xi)-\tilde G^{-1}(p)))}^1 \\ 
& \qquad \qquad  \frac{\phi(z)}{1-\max(n(\hat y_n(\xi)-y(\xi)), n(\hat y_n(\xi)-\tilde G^{-1}(p)))}dz\, dp\\
& \geq \frac{n}{1-n(\hat y_n(\xi)-y(\xi))} \int_{\tilde G(\hat y_n(\xi)-\frac1n)}^0 \int_{\max(n(\hat y_n(\xi)-y(\xi)), n(\hat y_n(\xi)- \tilde G^{-1}(p)))}^1 \phi(z) dz\,dp\\
& = \frac{n}{1-n(\hat y_n(\xi)-y(\xi))}  \int_{n(\hat y_n(\xi)-y(\xi))}^1 \int_{\tilde G(\hat y_n(\xi)-\frac{z}{n})}^0 \phi(z)dp\,dz\\
& = -\frac{n}{1-n(\hat y_n(\xi)-y(\xi))} \int_{n(\hat y_n(\xi)-y(\xi))}^1  \phi(z)\tilde G\Bigl(\hat y_n(\xi)-\frac{z}{n}\Bigr)dz.
\end{align*}
In the last step we used once more that both the area of integration and the integrand are bounded, which justifies once more the interchange of the order of integration. Thus we showed, so far, that 
\begin{align*}
\hat G_n'(\hat y_n(\xi))& \geq n\int_{n(\hat y_n(\xi)-y(\xi))}^1 \phi'(z)\tilde G\Bigl(\hat y_n(\xi)-\frac{z}{n}\Bigr)dz\\
& \geq -\frac{n}{1-n(\hat y_n(\xi)-y(\xi))} \int_{n(\hat y_n(\xi)-y(\xi))}^1 \phi(z) \tilde G\Bigl(\hat y_n(\xi)-\frac{z}{n}\Bigr)dz.
\end{align*}
The last step towards \eqref{goal1} is to replace the interval of integration $[n(\hat y_n(\xi)-y(\xi)),1]$ by $[-1, n(\hat y_n(\xi)-y(\xi))]$ and to use \eqref{help:est}. To that end observe that we have 
\begin{equation} \label{help:0}
\int_{-1}^1 \phi(z)G\Bigl(\hat y_n(\xi)-\frac{z}{n}\Bigr)dz=\hat G_n(\hat y_n(\xi))=\xi=\int_{-1}^1 \phi(z)\xi dz \quad \text{ for all } \xi\in \Real.
\end{equation}
Since $\tilde G(\hat y_n(\xi)-\frac{z}{n})=G(\hat y_n(\xi)-\frac{z}{n})-\xi$, we have 
\begin{equation*}
\int_{-1}^{n(\hat y_n(\xi)-y(\xi))} \phi(z) \tilde G\Bigl(\hat y_n(\xi)-\frac{z}{n}\Bigr)dz = -\int_{n(\hat y_n(\xi)-y(\xi))}^1 \phi(z) \tilde G\Bigl(\hat y_n(\xi)-\frac{z}{n}\Bigr)dz,
\end{equation*}
and, accordingly,
\begin{align*}
\hat G_n'(\hat y_n(\xi))& \geq -\frac{n}{1-n(\hat y_n(\xi)-y(\xi))}\int_{n(\hat y_n(\xi)-y(\xi))}^1 \phi(z)\tilde G\Bigl(\hat y_n(\xi)-\frac{z}{n}\Bigr)dz\\
& =\frac{n}{1-n(\hat y_n(\xi)-y(\xi))} \int_{-1}^{n(\hat y_n(\xi)-y(\xi))} \phi(z)\tilde G\Bigl(\hat y_n(\xi)-\frac{z}{n}\Bigr)dz\\
& \geq \frac{n}{1-n(\hat y_n(\xi)-y(\xi))}\int_{-1}^0 \phi(z)\frac{\hat y_n(\xi)-\frac{z}{n}-y(\xi)}{\varepsilon}dz\\ 
& \geq \frac{n(\hat y_n(\xi)-y(\xi))}{1-n(\hat y_n(\xi)-y(\xi))}\frac1{\varepsilon}\int_{-1}^0 \phi(z)dz\\
& \geq \frac{1}{2\varepsilon},
\end{align*}
where we used in the last step that $\frac12 \leq n(\hat y_n(\xi)-y(\xi))\leq 1$.
This finishes the proof of \eqref{goal1}. 

{\bf (b) The case $y_\xi(\xi)=c>0$.} 
By assumption $y_\xi(\xi)=c>0$ and hence for every $\varepsilon>0$ there exists some $\delta_\varepsilon>0$ such that 
\begin{equation}\label{est:quot2}
(1-\varepsilon)c<\frac{y(\eta)-y(\xi)}{\eta-\xi}<(1+\varepsilon)c \quad \text{ whenever } \vert \eta-\xi\vert \leq \delta_\varepsilon.
\end{equation}
Let $0<\varepsilon<1$ be fixed and define $\gamma_\varepsilon:=(1-\varepsilon)c\delta_\varepsilon$. In addition, let $z\in \Real $ be such that $\vert y(\xi)-z\vert <\gamma_\varepsilon$. We will first show that $\vert G(z)-\xi\vert <\delta_\varepsilon$. Indeed, assume the opposite. Then, due to \eqref{est:quot2}, if $G(z)\geq \xi+\delta_\varepsilon$, we have  
\begin{equation*}
z=y(G(z))\geq y(\xi+\delta_\varepsilon)\geq y(\xi)+\delta_\varepsilon(1-\varepsilon)c=y(\xi)+\gamma_\varepsilon,
\end{equation*}
and, if $G(z)\leq \xi-\delta_\varepsilon$, then
\begin{equation*}
z=y(G(z))\leq y(\xi-\delta_\varepsilon)\leq y(\xi)-\delta_\varepsilon(1-\varepsilon)c= y(\xi)-\gamma_\varepsilon.
\end{equation*}
Together, these estimates contradict $\vert y(\xi)-z\vert <\gamma_\varepsilon$, and hence prove that $\vert G(z)-\xi\vert <\delta_\varepsilon$.

As an immediate consequence, we obtain 
\begin{equation*}
(1-\varepsilon)c<\frac{y(G(z))-y(\xi)}{G(z)-\xi}<(1+\varepsilon)c \quad \text{ whenever }\vert z-y(\xi)\vert <\gamma_\varepsilon,
\end{equation*}
and thus, as $z=y(G(z))$ for all $z\in \Real$, 
\begin{equation}\label{goal22}
\frac1{(1+\varepsilon)c}<\frac{G(z)-\xi}{z-y(\xi)}<\frac{1}{(1-\varepsilon)c} \quad \text{ whenever } \vert z-y(\xi)\vert <\gamma_\varepsilon.
\end{equation}

In view of the above inequality \eqref{goal22}, which will play a key role, we assume without loss of generality that $\frac{2}{n}<\gamma_\varepsilon$ for the rest of this subsection.

The other main ingredient is to establish that $\lim_{n\to \infty} \vert n(\hat y_n(\xi)-y(\xi))\vert =0 $.
We note here that this fast convergence of $\hat y_n(\xi)$ to $y(\xi)$ need not necessarily
hold in points $\xi$ where $y_\xi(\xi) = 0$, cf.~the remark after this proof. 
We will only consider the case $\hat y_n(\xi)\leq y(\xi)$ and leave the other case, which follows the same lines, to the interested reader.
From \eqref{help:0}, we can deduce that 
\begin{align*}
0 &= \int_{-1}^1 \phi(z) \tilde G\Bigl(\hat y_n(\xi)-\frac{z}{n}\Bigr)dz\\ 
& =n\int_{\hat y_n(\xi)-\frac1n}^{\hat y_n(\xi)+\frac1n}  \phi(n(\hat y_n(\xi)-z))\tilde G(z)dz\\
& = n\int_{\hat y_n(\xi)-\frac1n}^{y(\xi)} \phi(n(\hat y_n(\xi)-z))\tilde G(z)dz\\
& \quad + n\int_{y(\xi)}^{\hat y_n(\xi)+\frac1n} \phi(n(\hat y_n(\xi)-z))\tilde G(z)dz\\
& \leq n\int_{\hat y_n(\xi)-\frac1n}^{y(\xi)} \phi(n(\hat y_n(\xi)-z))\frac{z-y(\xi)}{(1+\varepsilon)c}dz\\ 
& \quad + n\int_{y(\xi)}^{\hat y_n(\xi)+\frac1n} \phi(n(\hat y_n(\xi)-z))\frac{z-y(\xi)}{(1-\varepsilon)c}dz\\
& = n\int_{\hat y_n(\xi)-\frac1n}^{\hat y_n(\xi)+\frac1n} \phi(n(\hat y_n(\xi)-z))\frac{z-y(\xi)}{(1+\varepsilon)c} dz\\
& \quad + n\int_{y(\xi)}^{\hat y_n(\xi)+\frac1n} \phi(n(\hat y_n(\xi)-z))\frac{z-y(\xi)}{c}\left[ \frac{1}{1-\varepsilon}-\frac{1}{1+\varepsilon}\right]dz\\ 
& = \frac{\hat y_n(\xi)-y(\xi)}{(1+\varepsilon)c}+n\frac{2\varepsilon}{(1-\varepsilon^2)c}\int_{y(\xi)}^{\hat y_n(\xi)+\frac1n} \phi(n(\hat y_n(\xi)-z))(z-y(\xi))dz\\
& \leq \frac{\hat y_n(\xi)-y(\xi)}{(1+\varepsilon)c}+n\frac{2\varepsilon}{(1-\varepsilon^2)c}\int_{\hat y_n(\xi)}^{\hat y_n(\xi)+\frac1n}\phi(n(\hat y_n(\xi)-z))\frac2n dz\\ 
& \leq \frac{\hat y_n(\xi)-y(\xi)}{(1+\varepsilon)c}+\frac{2\varepsilon}{n(1-\varepsilon^2)c},
\end{align*}
where we used \eqref{goal22}.
Thus
\begin{equation}\label{help:estconvyny}
0\leq y(\xi)-\hat y_n(\xi)\leq \frac{2\varepsilon}{n(1-\varepsilon)},
\end{equation}
which implies that $\lim_{n\to\infty} n (y(\xi)-\hat y_n(\xi))=0$.

Let us return to the term $\hat G_n'(\hat y_n(\xi))$. We have from \eqref{def:Gpr} that
\begin{align*}
\hat G_n'(\hat y_n(\xi))& =n^2 \int_{\hat y_n(\xi)-\frac1n}^{\hat y_n(\xi)+\frac1n} \phi'(n(\hat y_n(\xi)-z))\tilde G(z)dz\\
& = n^2 \int_{\hat y_n(\xi)-\frac1n}^{\hat y_n(\xi)+\frac1n} \phi'(n(\hat y_n(\xi)-z))\frac{z-y(\xi)}{c}dz\\ 
& \quad + n^2 \int_{\hat y_n(\xi)-\frac1n}^{\hat y_n(\xi)+\frac1n} \phi'(n(\hat y_n(\xi)-z))\left[\tilde G(z)-\frac{z-y(\xi)}{c}\right]dz\\ 
& = \frac1c +n^2\int_{\hat y_n(\xi)-\frac1n}^{\hat y_n(\xi)+\frac1n}\phi'(n(\hat y_n(\xi)-z))\left[\tilde G(z)-\frac{z-y(\xi)}{c}\right]dz.
\end{align*}
Thus \eqref{goal2} will follow if we can show that the last term on the right-hand side tends to $0$ as $n\to \infty$. Now observe that \eqref{goal22} implies that 
\begin{equation*}
\Bigl\lvert \tilde G(z)-\frac{z-y(\xi)}{c}\Bigr\rvert =\frac{\lvert z-y(\xi)\rvert}{c} \Bigl\lvert 1-\frac{c \tilde G(z)}{z-y(\xi)}\Bigr\rvert \leq \frac{\varepsilon}{(1-\varepsilon)c}\lvert z-y(\xi)\rvert,
\end{equation*} 
and hence 
\begin{align*}
\Bigl\lvert n^2 \int_{\hat y_n(\xi)-\frac1n}^{\hat y_n(\xi)+\frac1n}&  \phi'(n(\hat y_n(\xi)-z))\Bigl[\tilde G(z)-\frac{z-y(\xi)}{c}\Bigr]dz\Bigr\rvert \\ 
 & \leq \frac{n^2}{c} \frac{\varepsilon}{1-\varepsilon}\int_{\hat y_n(\xi)-\frac1n}^{\hat y_n(\xi)+\frac1n} \lvert \phi'(n(\hat y_n(\xi)-z))\rvert \lvert z-y(\xi)\rvert dz\\
 & \leq n^2\frac{\varepsilon}{(1-\varepsilon)c} \int_{\hat y_n(\xi)-\frac1n}^{\hat y_n(\xi)+\frac1n} \phi'(n(\hat y_n(\xi)-z))(z-y(\xi)) dz\\
& \quad +2 n^2\frac{\varepsilon}{(1-\varepsilon)c}\int_{I_0} \lvert \phi'(n(\hat y_n(\xi)-z))\rvert \lvert z-y(\xi)\rvert dz\\
& =\frac{\varepsilon}{(1-\varepsilon)c}\left(1 + 2 n^2\int_{I_0} \lvert \phi'(n(\hat y_n(\xi)-z))\rvert \lvert z-y(\xi)\rvert dz\right),
\end{align*}
where we used \eqref{goal11}. 

Recall from \eqref{def:I0} that for $z\in I_0$ one has
\begin{equation*}
\hat y_n(\xi)-\vert \hat y_n(\xi)-y(\xi)\vert \leq z\leq \hat y_n(\xi)+\vert \hat y_n(\xi)-y(\xi)\vert 
\end{equation*}
and therefore
\begin{equation*}
\vert z-y(\xi)\vert \leq 2\vert \hat y_n(\xi)-y(\xi)\vert \leq \frac{4\varepsilon}{n(1-\varepsilon)},
\end{equation*}
where we used \eqref{help:estconvyny}.

This implies that 
\begin{align*}
\Bigl\lvert \hat G_n'(\hat y_n(\xi))-\frac1c\Bigr\rvert &  \leq \Bigl\lvert n^2\int_{\hat y_n(\xi)-\frac1n}^{\hat y_n(\xi)+\frac1n} \phi'(n(\hat y_n(\xi)-z))\Bigl[\tilde{G}(z)-\frac{z-y(\xi)}{c}\Bigr]dz\Bigr\rvert\\
& \leq \frac{\varepsilon}{(1-\varepsilon)c}\left(1+\frac{8\varepsilon}{(1-\varepsilon)} n\int_{I_0} \lvert \phi'(n(\hat y_n(\xi)-z))\rvert dz\right)\\
& \leq \frac{\varepsilon}{(1-\varepsilon)c}\left(1+\frac{16\varepsilon}{1-\varepsilon} \phi(0)\right).
\end{align*}

Since $\varepsilon>0$ can be chosen arbitrarily small, this implies that $\hat G_n'(\hat y_n(\xi))\to \frac1c$ as $n\to \infty$.

To summarise, we have in this step shown that $\hat{G}_n'(\hat{y}_n(\xi))$ converges to $1/y_\xi(\xi)$
in every point $\xi \in \Real$ where $y$ is differentiable. Thus also $\tilde{h}_n(\xi)$ converges
to $h(\xi)$ in all of these points. Together with the fact that $\lVert\tilde{h}_n\rVert_{L^1} =\lVert h\rVert_{L^1}$ 
for all $n$ (see~\eqref{eq:tildehnynNorm}), this shows that
$\lVert \tilde{h}_n-h\rVert_{L^1} \to 0$.

{\bf Step 4. Convergence of $\tilde h_n$ to $h$ in $L^2(\Real)$.}
Recall that 
\begin{equation}\label{eq:1-}
\tilde h_n=1-\hat y_{n,\xi} \quad \text{ and }\quad  h=1-y_\xi.
\end{equation}
Since  $\tilde h_n$, $\hat y_{n,\xi}$, $h$, and $y_\xi$ all are non-negative, it follows that 
\begin{equation*}
0\leq \tilde h_n(\xi)\leq 1\quad \text{ and } \quad 0\leq h(\xi)\leq 1 \quad \text{ for all }\xi\in\Real.
\end{equation*}
Thus we have 
\begin{equation*}
\bnorm{\tilde h_n-h}_{L^2}^2\leq \bnorm{\tilde h_n-h}_{L^1}.
\end{equation*}
Since we already know that $\tilde h_n\to h$ in $L^1(\Real)$, the claim follows.

{\bf Step 5. Convergence of $\hat y_{n,\xi}-1$ to $y_\xi-1$ in $L^1(\Real)\cap L^2(\Real)$.}
By definition we have 
\begin{equation*}
\tilde h_n=1-\hat y_{n,\xi}\quad \text{ and } \quad h=1-y_\xi.
\end{equation*}
Since $\tilde h_n\to h$ both in $L^1(\Real)$ and $L^2(\Real)$ the claim follows.

{\bf Step 6. Convergence of $\hat U_n$ to $U$ in $L^2(\Real)$.}
A proof can be found in \cite[Prop.~5.1]{HR}. 

{\bf Step 7. Convergence of $\hat U_{n,\xi}$ to $U_\xi$ in $L^2(\Real)$.}

Let $S=\{\xi\in\Real \mid y_\xi(\xi)=0\}$. Then $U_\xi(\xi)=0$ for almost all $\xi\in S$, 
since $U^2_\xi=hy_\xi$ almost everywhere. Thus we have
\begin{equation}\label{eq:11}
 \bnorm{\hat U_{n,\xi}-U_\xi}_{L^2}^2=\int_S \hat U_{n,\xi}^2(\xi)d\xi+\int_{S^c} (\hat U_{n,\xi}-U_\xi)^2(\xi)d\xi.
\end{equation}
From \eqref{eq:tildehnyn} and the fact that $\bar{\hat{r}}_n\geq 0$,
it follows that we have for almost every $\xi\in S$ that
\begin{equation*}
\hat U_{n,\xi}^2(\xi)\leq \tilde h_n(\xi)\hat y_{n,\xi}(\xi)=\tilde h_n(\xi)(\hat y_{n,\xi}-y_\xi)(\xi)=\tilde h_n(\xi)(h-\tilde h_n)(\xi),
\end{equation*}
and, therefore, 
\begin{equation*}
 \int_S \hat U_{n,\xi}^2(\xi)d\xi\leq \bnorm{h-\tilde h_n}_{L^1},
\end{equation*}
since $\bnorm{\tilde h_n}_{L^\infty}\le1$. Thus the first integral in \eqref{eq:11} tends to $0$ as $n\to\infty$.

As far as the integral over $S^c$ is concerned, the proof of the convergence follows closely the one of $\bar r_n\to \bar r$ in $L^2(\Real)$ as $n\to \infty$ in \cite[Lemma 6.4]{GHR4}, which we reproduce here for completeness. Note that by definition we have  $\hat U_{n,\xi}(\xi)=u_{n,x}(\hat y_n(\xi))\hat y_{n,\xi}(\xi)$ and 
$U_\xi(\xi)=u_x(y(\xi))y_\xi(\xi)$ for almost every $\xi\in S^c$, so that 
\begin{align}\label{est:stab2}
 \norm{U_{n,\xi}-U_\xi}_{L^2(S^c)}^2 & = \norm{(u_{n,x}\circ \hat{y}_n)\hat{y}_{n,\xi}-(u_x\circ y)y_\xi}_{L^2(S^c)}^2\\  \nn
& = \int_{S^c} (u_{n,x}\circ \hat{y}_n)^2\hat{y}_{n,\xi}(\hat{y}_{n,\xi}-y_\xi)d\xi\\ \nn
&\quad +\int_{S^c} (u_{n,x}\circ \hat{y}_n)\hat{y}_{n,\xi}(u_{n,x}\circ \hat{y}_n-u_{n,x}\circ y)y_\xi d\xi\\ \nn
&\quad +\int_{S^c} (u_{n,x}\circ \hat{y}_n)\hat{y}_{n,\xi}(u_{n,x}\circ y-u_x\circ y)y_\xi d\xi\\ \nn
&\quad +\int_{S^c} (u_x\circ y)^2y_\xi(y_\xi-\hat{y}_{n,\xi})d\xi\\ \nn
&\quad +\int_{S^c} (u_x\circ y)y_\xi(u_x\circ y-u_x\circ \hat{y}_n)\hat{y}_{n,\xi}d\xi\\ \nn
&\quad +\int_{S^c} (u_x\circ y)y_\xi(u_x\circ \hat{y}_n-u_{n,x}\circ \hat{y}_n)\hat{y}_{n,\xi}d\xi.
\end{align}
The first and the fourth term have the same structure, and we therefore only treat the first one. Since  $(u_{n,x}\circ \hat{y}_n)^2\hat{y}_{n,\xi}\leq \tilde h_n\leq 1$, we have 
\begin{equation*}
\norm{(u_{n,x}\circ \hat{y}_n)^2\hat{y}_{n,\xi}( \hat{y}_{n,\xi}-y_{\xi})}_{L^1(S^c)}\leq \norm{y_{\xi}-\hat{y}_{n,\xi}}_{L^1}
\end{equation*}
and thus this term tends to $0$ as $n\to\infty$. 
In order to investigate the fifth term we will use that $u_x\in L^2(\Real)$ and therefore there exists for any $\varepsilon>0$ a continuous function $\tilde l$ with compact support such that $\bnorm{u_x-\tilde l}_{L^2}\leq \varepsilon /(3\max(1,\norm{u_x}_{L^2}))$.
Thus we can write
\begin{multline}\label{eq:stab2_fifth}
 \bnorm{(u_x\circ y)y_{\xi}(u_x\circ y-u_x\circ \hat{y}_n)\hat{y}_{n,\xi}}_{L^1(S^c)}\\
\begin{aligned}
  &\leq \bnorm{(u_x\circ y)y_{\xi}(u_x\circ y-\tilde  l\circ y)\hat{y}_{n,\xi}}_{L^1(S^c)}\\ 
  & \quad +\bnorm{(u_x\circ y)y_{\xi}(\tilde l\circ y-\tilde l\circ \hat{y}_n)\hat{y}_{n,\xi}}_{L^1(S^c)}\\
  &\quad +\bnorm{(u_x\circ y)y_{\xi}(\tilde l\circ \hat{y}_n-u_x\circ \hat{y}_n)\hat{y}_{n,\xi}}_{L^1(S^c)} \\ 
  & \leq \norm{u_x}_{L^2}\Bigl(2 \bnorm{u_x-\tilde l}_{L^2}+\bnorm{\tilde l\circ \hat{y}_n-\tilde l\circ y}_{L^2}\Bigr).
\end{aligned}
\end{multline}
Here we have used in the last inequality that both $y_\xi$ and $\hat{y}_{n,\xi}$ are non-negative and
bounded above by 1.
 Since $\hat{y}_n - y \to 0$ in $L^\infty(\Real)$ and $\tilde l$ is continuous with compact support, we obtain by Lebesgue's dominated convergence theorem that $\tilde l\circ \hat{y}_n\to \tilde l\circ y$ in $L^2(\Real)$. In particular, we can choose $n$ large enough so that $\norm{(u_x\circ y)y_{\xi}(u_x\circ y-u_x\circ \hat{y}_n)\hat{y}_{n,\xi}}_{L^1(S^c)}\leq \varepsilon$. Since $\varepsilon$ can be chosen arbitrarily small, we  obtain in particular that  
\begin{equation}
 \lim_{n\to\infty}\bnorm{(u_x\circ y)y_{\xi}(u_x\circ y-u_x\circ \hat{y}_n)\hat{y}_{n,\xi}}_{L^1(S^c)}=0.
\end{equation}
For the convergence of the second term, we estimate, using again that $y_\xi$ is bounded by 1,
\begin{multline*}
\bnorm{(u_{n,x}\circ \hat{y}_n)\hat{y}_{n,\xi}(u_{n,x}\circ \hat{y}_n-u_{n,x}\circ y)y_\xi}_{L^1(S^c)} \\
\begin{aligned}
&\le \bnorm{(u_{n,x}\circ \hat{y}_n)\hat{y}_{n,\xi}(u_{n,x}\circ \hat{y}_n-u_x\circ \hat{y}_n)y_\xi}_{L^1(S^c)} \\
&\quad+\bnorm{(u_{n,x}\circ \hat{y}_n)\hat{y}_{n,\xi}(u_x\circ \hat{y}_n-u_x\circ y)y_\xi}_{L^1(S^c)} \\
&\quad+\bnorm{(u_{n,x}\circ \hat{y}_n)\hat{y}_{n,\xi}(u_x\circ y-u_{n,x}\circ y)y_\xi}_{L^1(S^c)} \\
&\le \Bigl(\int_{S^c} (u_{n,x}\circ \hat{y}_n)^2 \hat{y}_{n,\xi}\,d\xi\Bigr)^{1/2}
\Bigl(\int_{S^c} (u_{n,x}\circ\hat{y}_n-u_x\circ\hat{y}_n)^2\hat{y}_{n,\xi}\,d\xi\Bigr)^{1/2}\\
&\quad+ \bnorm{(u_{n,x}\circ \hat{y}_n)\hat{y}_{n,\xi}(u_x\circ \hat{y}_n-u_x\circ y)y_\xi}_{L^1(S^c)}\\
&\quad+ \bnorm{(u_{n,x}\circ \hat{y}_n)\hat{y}_{n,\xi}}_{L^2}\bnorm{(u_x\circ y - y_{n,x}\circ y)y_\xi}_{L^2}.
\end{aligned}
\end{multline*}
The first and third term in this last estimate tend to zero because $u_{n,x} \to u_x \in L^2(\Real)$
and both $y_\xi$ and $\hat{y}_{n,\xi}$ are uniformly bounded, and for
the convergence of the second term we can use the same method as in~\eqref{eq:stab2_fifth}.
Thus also the second term in~\eqref{est:stab2} tends to zero.
As far as the third (and, similarly, the last) term in~\eqref{est:stab2} is concerned, we have that
\begin{align*}
& \norm{(u_{n,x}\circ \hat{y}_n)\hat{y}_{n,\xi}(u_{n,x}\circ y-u_x\circ y)y_\xi}_{L^1(S^c)}\\
& \qquad \qquad \leq \norm{(u_{n,x}\circ \hat{y}_n) \hat{y}_{n,\xi}}_{L^2(S^c)}\norm{(u_{n,x}\circ y-u_x\circ y)y_\xi}_{L^2(S^c)}\\
& \qquad \qquad \leq \norm{u_{n,x}}_{L^2}\norm{u_{n,x}-u_x}_{L^2},
\end{align*}
which again tends to zero since by assumption $u_{n,x}\to  u_x\in L^2(\Real)$.
Hence all terms in \eqref{est:stab2} tend to $0$ as $n\to\infty$ and therefore $\hat U_{n,\xi}\to U_\xi$ in $L^2(\Real)$.
 
{\bf Step 8. Convergence of $\bar{\hat {r}}_n$ to zero in $L^2(\Real)$.} By construction, we have $\bar {\hat{r}}_n\geq 0$ since 
  $\bar \rho_n\geq 0$ and $\hat {y}_{n,\xi}\geq 0$. Hence, by \eqref{eq:lagcoord3}, \eqref{eq:Ldef2}, \eqref{eq:tildehnyn}, and \eqref{eq:1-}, we have 
  \begin{align*}
 \norm{\bar{\hat{ r}}_n}_{L^2}^2
 & \leq \norm{\bar{\hat{r}}_n^2+\frac2n \bar{\hat{r}}_n\hat y_{n,\xi}}_{L^1}\\ \nn  
 & = \bnorm{\tilde h_n\hat y_{n,\xi} -hy_\xi-\hat U^2_{n,\xi}+U^2_\xi}_{L^1}\\ \nn
& =\bnorm{\tilde h_n-h-\tilde h_n^2+h^2-\hat U_{n,\xi}^2+U_\xi^2}_{L^1}\\ \nn 
& \leq 3\bnorm{h-\tilde h_n}_{L^1}+\bigl(\bnorm{\hat U_{n,\xi}}_{L^2}+\norm{U_{\xi}}_{L^2}\bigr)\bnorm{\hat U_{n,\xi}-U_\xi}_{L^2}\\ \nn
& \leq 3\bnorm{h-\tilde h_n}_{L^1}+2\norm{\mu}\bnorm{\hat U_{n,\xi}-U_\xi}_{L^2}.
\end{align*}
Since $\tilde h_n\to h$ in $L^1(\Real)$ and $\hat U_{n,\xi}\to U_\xi$ in $L^2(\Real)$, the above estimate implies that $\bar{\hat{r}}_n\to 0$ in $L^2(\Real)$ as $n\to \infty$.
  
{\bf Step 9. Convergence of $\hat h_n$ to $h$ in $L^1(\Real)\cap L^2(\Real)$.}
According to \eqref{eq:Ldef2} and \eqref{eq:Ldef2neu}, we have 
\begin{equation*}
\bnorm{\hat h_n-h}_{L^2}\leq \frac2n\norm{\bar{\hat{r}}_n}_{L^2}+\norm{\hat y_{n,\xi}-y_\xi}_{L^2}.
\end{equation*}
Since $\bar{\hat{r}}_n\to 0$ and $\hat y_{n,\xi}-y_\xi\to 0$ in $L^2(\Real)$, this inequality implies that $\hat h_n\to h$ in $L^2(\Real)$. 
As far as the $L^1(\Real)$ convergence is concerned, observe that
\begin{align*}
\hat h_n-h&= \hat h_n (\hat y_{n,\xi}+\tilde h_n)-h(y_\xi+h)\\
& =\hat h_n\hat y_{n,\xi}-hy_\xi +\hat h_n\tilde h_n-h^2 \\
& = \hat U_{n,\xi}^2-U_\xi^2+\bar{\hat{r}}_n^2+\hat h_n \tilde h_n-h^2.
\end{align*}
Here the first equality follows from~\eqref{eq:Ldef2} and~\eqref{eq:Ldef5neu},
and the last equality follows from~\eqref{eq:lagcoord3} and~\eqref{eq:tildehnyn}.
Thus
\begin{align*}
\bnorm{\hat h_n-h}_{L^1}&\leq \bnorm{\hat U_{n,\xi}+U_\xi}_{L^2}\bnorm{\hat U_{n,\xi}-U_\xi}_{L^2}+\norm{\bar{\hat{r}}_n}_{L^2}^2\\ \nn
&\quad +\bnorm{\tilde h_n}_{L^2}\bnorm{\hat h_n-h}_{L^2}+\norm{h}_{L^2}\bnorm{\tilde h_n-h}_{L^2},
\end{align*}
which implies that $\hat h_n\to h$ in $L^1(\Real)$. 

{\bf Step 10. Convergence of $X_n$ to $X$ in $E$.} So far we have shown that $\hat X_n\to X$ in $E$. 
In addition we showed in Step 1 for all $n\in \Natural$ that we can write $\hat X_n=X_n\circ g_n$
with $g_n\in G_{\kappa}$ for some $\kappa > 0$ independent of $n$, or, equivalently
that $\hat X_n \in \F_\kappa$ for all $n \in \Natural$.
Moreover, it is known, see, e.g., \cite[Lemma 4.6]{GHR4}, that the mapping
$\Gamma \colon \F_\kappa \to \F_0$ defined in~\eqref{eq:Gammadef} is continuous.
Thus we also have that $X_n = \Gamma(\hat{X}_n) \to X$ in $E$, which completes the proof.
\end{proof}

\begin{remark}
A closer look at the proof of Theorem~\ref{thm:approxLagran} reveals that we showed that for every $\xi\in \Real$ 
where $y$ is differentiable and $y_\xi>0$, we have 
\begin{equation*}
\lim_{n\to\infty} n(\hat y_n(\xi)-y(\xi))=0.
\end{equation*}
As the following example illustrates, we cannot expect this convergence to hold for almost every $\xi\in\Real$ such that $y_\xi(\xi)=0$.

Consider the following initial data for the CH equation which corresponds to a symmetric/antisymmetric peakon-antipeakon solution, which vanishes at breaking time $t=0$, i.e., 
\begin{equation*}
u(x) = 0 \text{ for all } x \in \Real \quad \text{ and } \quad F(x)=\mu((-\infty,x))=\begin{cases} 0, \quad x<0,\\
\alpha, \quad 0\leq x,
\end{cases}
\end{equation*}
where $\alpha>0$.
Then 
\begin{equation*}
y(\xi)=\begin{cases} 
\xi, & \xi<0,\\
0, & 0\leq\xi\leq\alpha,\\
\xi-\alpha, & \alpha<\xi.
\end{cases}  
\end{equation*}
and especially $y_\xi(\xi)=0$ for all $\xi\in(0,\alpha)$.
For the approximating sequence we know that 
\begin{equation*}
\hat y_n(\xi)+\hat F_n(\hat y_n(\xi))=\xi \quad \text{ for all }\xi\in\Real, 
\end{equation*}
where 
\begin{equation*}
\hat F_n(x)=\int_\Real n\phi(n(x-y))F(y)dy.
\end{equation*}
We are going to show that $\lim_{n\to\infty}n(\hat y_n(\xi)-y(\xi))\not =0$ for any $\xi\in (0,\alpha)$ except $\xi=\frac{\alpha}2$.

Indeed, if we denote
\[
\Phi(x) := \int_{-\infty}^x \phi(y)\,dy,
\]
then we see that
\[
\hat{F}_n(x) = \int_\Real n\phi(n(x-y))F(y)dy
= \alpha\int_0^\infty n\phi(n(x-y))dy
= \alpha\Phi(nx)
\]
for all $x \in \Real$.
Now assume that $0 < \xi < \alpha$.
Then $y(\xi) = 0$ and thus
\[
\xi = \hat{y}_n(\xi) + \hat{F}_n(\hat{y}_n(\xi))
= \hat{y}_n(\xi) + \alpha\Phi(n\hat{y}_n(\xi))
= \hat{y}_n(\xi)-y(\xi) + \alpha\Phi(n(\hat{y}_n(\xi)-y(\xi))).
\]
In Step 2 in the proof of Theorem~\ref{thm:approxLagran}
we have shown that $\hat{y}_n(\xi) \to y(\xi)$.
Taking the limit $n\to \infty$ in the previous equation
therefore implies that
\[
\frac{\xi}{\alpha} = \lim_{n\to\infty} \Phi(n(\hat{y}_n(\xi)-y(\xi))).
\]
Using that $\abs{n(\hat{y}_n(\xi)-y(\xi))} < 1$ for all $n$
and that $\Phi$ is continuously invertible on $(-1,1)$,
it follows that
\[
\lim_{n\to\infty} n(\hat{y}_n(\xi)-y(\xi)) = \Phi^{-1}(\xi/\alpha).
\]
Since $\Phi(0) = 1/2$ and therefore $\Phi^{-1}(1/2) = 0$,
this shows in particular that
the sequence $n(\hat{y}_n(\xi)-y(\xi))$ only converges
to $0$ for $\xi = \alpha/2$.
\end{remark}

\section{Convergence in Lagrangian coordinates implies convergence in Eulerian coordinates} \label{sec:4}

In the previous two sections, we saw that we can approximate any given initial data
$(u,\mu)$ for the CH equation by a sequence of smooth initial data $(u_n,\rho_n,\mu_n)$ for the 2CH system where the measures $\mu_n$ are purely absolutely continuous. Afterwards we saw that this convergence in Eulerian coordinates is transported, via the mapping $L$, to convergence in Lagrangian coordinates. 

In this section we consider the case when we are given a sequence $X_n\in \F_0$ and $X\in \F_0$, such that $X_n\to X$ in $E$. Does $M(X_n)\to M(X)$ in some sense in Eulerian coordinates? Here $M\colon\F_0\to \D$ denotes the mapping from Lagrangian to Eulerian coordinates, which is defined as follows. 

\begin{definition}[{\cite[Thm.~4.10]{GHR4}}]
Given any element $X=(y,U,h,r)\in \F_0$, we define $(u,\rho,\mu)$ as follows\footnote{Here we denote by $y_\#(h\,d\xi)$ the \emph{push-forward} of the measure $h\,d\xi$ by $y$, defined by
$y_\#(h\,d\xi)(A) = \int_{y^{-1}(A)} h(\xi)d\xi$ for all Borel sets $A \subset \Real$.}
\begin{subequations}
\label{eq:umudef}
\begin{align}
\label{eq:umudef1}
u(x)&=U(\xi)\text{ for any }\xi\text{ such that  }  x=y(\xi),\\
\label{eq:umudef2}
\mu&=y_\#(h(\xi)\,d\xi),\\
\label{eq:umudef3}
\bar \rho(x)\,dx&=y_\#(\bar r(\xi)\,d\xi),\\
\label{eq:umudef4}
\rho(x)&=k+\bar\rho(x),
\end{align}
\end{subequations}
where $k$ is implicitly given through the relation $r(\xi)=\bar r(\xi)+ky_\xi(\xi)$ for all $\xi\in\Real$.
We have that $(u,\rho,\mu)$ belongs to $\D$ and, in particular, that the measure $y_{\#}(\bar r(\xi)\, d\xi)$ is absolutely continuous. We
denote by $M\colon \F_0\rightarrow\D$ the mapping
which to any $X$ in $\F_0$ associates the element $(u, \rho,\mu)\in \D$ as given
by \eqref{eq:umudef}.
\end{definition}

We recall from Definition~\ref{def:F}
that for any $(y,U,h,r) \in \F_0$ we have that $y_\xi \ge 0$
and $U_\xi = 0$ whenever $y_\xi = 0$, or, in other words, that $U$ is
constant whenever the increasing function $y$ is constant. As a
consequence, the value $U(\xi)$ is uniquely determined by $y(\xi)$,
which means that the definition of the function $u$
in~\eqref{eq:umudef1} is independent of the choice of $\xi$ satisfying
$x = y(\xi)$. 
Also, the fact that $y$ is Lipschitz continuous
(see~\eqref{eq:lagcoord1}) implies that the push-forward of the
absolutely continuous measure $\bar r(\xi)\,d\xi$ under $y$ is again 
absolutely continuous, cf. \cite[Thm.~4.10]{GHR4}.

Moreover, we consider the following notion of sequential convergence
on $\D$.

\begin{definition}\label{def:convD}
  We say that a sequence $(u_n,\rho_n,\mu_n)\in \D$ converges
  to $(u,\rho,\mu)\in\D$ as $n\to\infty$ if\footnote{We say that
    $f_n\rightharpoonup f$ if $\int_\Real f_n(x)g(x)dx\to \int_\Real
    f(x)g(x)dx $ for every $g\in L^2(\Real)$.}
  \begin{subequations}
    \begin{align}
      u_n&\to u \in L^2(\Real)\cap L^\infty(\Real),\\
      u_{n,x}& \rightharpoonup u_x,\\
      \bar \rho_n &\rightharpoonup\bar \rho,\\
      \label{assump:EulerLag4}
      k_n&\to k,\\
      \label{assump:EulerLag1}
      \int_\Real
      \frac{u_{n,x}^2(x)}{1+u_{n,x}^2(x)+\bar\rho_n^2(x)}dx&\to\int_\Real
                                                             \frac{u_x^2(x)}{1+u_{x}^2(x)+\bar\rho^2(x)}dx,\\
      \label{assump:EulerLag2}
      \int_\Real
      \frac{\bar\rho_{n}^2(x)}{1+u_{n,x}^2(x)+\bar\rho_n^2(x)}dx&\to
                                                                  \int_\Real
                                                                  \frac{\bar\rho^2(x)}{1+u_x^2(x)+\bar\rho^2(x)}dx,\\
      \label{assump:EulerLag3}
      F_n(x)&\to F(x)\text{ for every $x$ at which $F$ is continuous}, \\
      F_n(\infty)&\to F(\infty),
    \end{align}
  \end{subequations}
  where $F_n(x)=\mu_n((-\infty,x])$ for all $n\in \Natural$ and $F(x)=\mu((-\infty,x])$.  
\end{definition}

With this definition, the convergence result can be stated as follows.

\begin{theorem}\label{thm:LagEuler}
Given a sequence $X_n=(y_n,U_n,h_n,r_n)\in \F_0$ and $X=(y,U,h,r)\in
\F_0$ such that $X_n\to X$ in $E$ as $n\to\infty$, then
$(u_n,\rho_n,\mu_n)=M(X_n)$ converges to $(u,\rho,\mu)=M(X)$ in the
sense of Definition~\ref{def:convD}.
\end{theorem}

\begin{proof} The proof is divided into 8 steps for convenience.

{\bf Step 1. Convergence of $u_n$ to $u$ in $L^\infty(\Real)$.} For a proof we refer the interested reader to \cite[Thm.~6.5]{GHR4}.

{\bf Step 2.  Convergence of $u_n$ to $u$ in $L^2(\Real)$.} If we can show that the assumptions of the Radon--Riesz theorem are fulfilled, see, e.g., \cite[Thm.~1.37]{Ambrosio}, the claim follows. Thus we have to show that $\norm{u_n}_{L^2}\to\norm{u}_{L^2}$ and that $u_n$ converges weakly to $u$ as $n\to\infty$.

A straightforward computation using \eqref{eq:umudef1} yields
\begin{equation}\label{eq:17}
\norm{u_n}_{L^2}^2=\int_\Real (U_n^2-U^2)y_{n,\xi}(\xi)d\xi+\int_\Real U^2(y_{n,\xi}-y_\xi)(\xi)d\xi+\norm{u}_{L^2}^2,
\end{equation}
where we have used that $U^2y_\xi(\xi)=0$ whenever $y_\xi(\xi)=0$, and similarly that $U^2_ny_{n,\xi}(\xi)=0$ whenever $y_{n,\xi}(\xi)=0$. 
Applying the Cauchy--Schwarz inequality to the first and second term on the right-hand side of \eqref{eq:17} yields that 
\begin{align*}
\Bigl\lvert \int_\Real (U_n^2-& U^2) y_{n,\xi}(\xi)d\xi\Bigr\rvert + \Bigl\lvert \int_\Real U^2(y_{n,\xi}-y_\xi)(\xi)d\xi\Bigr\rvert \\
& \leq (\norm{U_n-U}_{L^2}+2\norm{U}_{L^2})\norm{U_n-U}_{L^2}+\norm{U}_{L^\infty}\norm{U}_{L^2}\norm{y_{n,\xi}-y_\xi}_{L^2},
\end{align*}
where we used that $0\leq y_{n,\xi}\leq 1$. Since $U_n\to U$ in
$L^2(\Real)$ and $\zeta_{n,\xi}\to \zeta_\xi$ in $L^2(\Real)$, we
obtain from~\eqref{eq:17} that $\norm{u_n}_{L^2}\to \norm{u}_{L^2}$ as
$n\to \infty$.

Since $C_c^\infty(\Real)$ is dense in $L^2(\Real)$ and $\norm{u_n}_{L^2}\to \norm{u}_{L^2}$, it suffices to show that 
\begin{equation}\label{eq:19}
 \lim_{n\to\infty} \int_\Real u_n\psi(x)dx=\int_\Real u\psi(x)dx
\end{equation}
for all $\psi\in C_c^\infty(\Real)$. 
This however follows immediately, as
\[
\abs{\int_\Real (u_n-u)\psi(x)dx}
\le \norm{u_n-u}_{L^\infty}\norm{\psi}_{L^1}
\to 0
\]
according to Step~1.

{\bf Step 3. Convergence of $F_n(x)$ to $F(x)$ for all $x$ at which $F(x)$ is continuous.} According to \cite[Props.~7.19 and 8.17]{Folland}, 
this is equivalent to showing that 
\begin{equation}\label{eq:21}
 \lim_{n\to\infty}\int_\Real \psi(x)d\mu_n(x)=\int_\Real \psi(x)d\mu(x)
\end{equation}
for all $\psi\in C_c^\infty(\Real)$. It follows from \eqref{eq:umudef2} that 
\begin{equation}\label{eq:22}
 \int_\Real \psi d\mu_n(x)=\int_\Real \psi(y_n(\xi))h_n(\xi)d\xi,
\end{equation}
and
\begin{equation}\label{eq:23}
 \int_\Real \psi d\mu(x)=\int_\Real \psi(y(\xi))h(\xi)d\xi.
\end{equation}
Since $y_n-\id\to y-\id$ in $L^\infty(\Real)$, the support of $\psi\circ y_n$ is contained in some compact set which can be chosen independently of $n$, and, from Lebesgue's dominated convergence theorem, we have that $\psi\circ y_n\to\psi\circ y$ in $L^2(\Real)$. Hence, since $h_n\to h$ in $L^2(\Real)$,
\begin{equation*}
 \lim_{n\to\infty} \int_\Real \psi(y_n(\xi))h_n(\xi)d\xi=\int_\Real \psi(y(\xi))h(\xi)d\xi,
\end{equation*}
and \eqref{eq:21} follows from \eqref{eq:22} and \eqref{eq:23}.

Note that, in particular, $\mu_n(\Real)\to\mu(\Real)$ as $n\to\infty$, since $\norm{h_n}_{L^1}\to \norm{h}_{L^1}$ by assumption. Moreover, 
\begin{equation*}
\mu_{n,{\rm ac}}=(u_{n,x}^2+\bar\rho_n^2)dx \quad \text{and} \quad \mu_{\rm ac}=(u_x^2+\bar\rho^2)dx,
\end{equation*}
which implies 
\begin{equation}\label{fellesest}
\norm{u_{n,x}}_{L^2}^2+\norm{\bar\rho_n}_{L^2}^2\leq\mu_n(\Real) \quad  \text{and} \quad \norm{u_x}_{L^2}^2+\norm{\bar\rho}_{L^2}^2\leq\mu(\Real),
\end{equation}
and hence $u_{n,x}$, $u_x$, $\bar\rho_n$ and $\bar\rho$ belong to $L^2(\Real)$.

{\bf Step 4. Weak convergence of $u_{n,x}$ to $u_x$.}
Since $C_c^\infty(\Real)$ is dense in $L^2(\Real)$ and $\norm{u_{n,x}}_{L^2}$ and $\norm{u_x}_{L^2}$ can be uniformly bounded according to \eqref{fellesest}, it suffices to show that 
\begin{equation*}
\lim_{n\to\infty} \int_\Real u_{n,x}(x)\psi(x)dx=\int_\Real u_x(x)\psi(x)dx
\end{equation*} 
for all $\psi\in C_c^\infty(\Real)$. 
To that end, observe that 
\begin{equation*}
\int_\Real u_{n,x}(x)\psi(x)dx =\int_\Real U_{n,\xi}(\xi)\psi(y_n(\xi))d\xi,
\end{equation*}
since $U_{n,\xi}(\xi)=0$ for all $\xi\in\Real$ such that $y_{n,\xi}(\xi)=0$,
and 
\begin{equation*}
\int_\Real u_x(x)\psi(x)dx=\int_\Real U_\xi(\xi)\psi(y(\xi))d\xi.
\end{equation*}
Thus it suffices to show that 
\begin{equation}\label{equiv:weak}
\lim_{n\to\infty} \int_\Real U_{n,\xi}(\xi)\psi(y_n(\xi))d\xi=\int_\Real U_\xi(\xi)\psi(y(\xi))d\xi,
\end{equation}
for all $\psi\in C_c^\infty(\Real)$. By assumption we have that $U_{n,\xi}\to U_\xi$ in $L^2(\Real) $ and $y_n-\id\to y-\id$ in $L^\infty(\Real)$ and hence the support of $\psi(y_n(\xi))$ and $\psi(y(\xi))$ is contained in some compact set that can be chosen independent of $n$. Thus $\psi(y_n(\xi))\to \psi(y(\xi))$ in $ L^\infty(\Real)\cap L^2(\Real)$, and Lebesgue's dominated convergence theorem implies \eqref{equiv:weak}.

{\bf Step 5. Weak convergence of $\bar\rho_n$ to $\bar\rho$.}
The argument closely follows the one of $u_{n,x}$ convergences weakly to $u_x$ in Step 4.

{\bf Step 6. Convergence of $\int_\Real \frac{u_{n,x}^2(x)}{1+u_{n,x}^2(x)+\bar\rho_n^2(x)}dx$ to $\int_\Real \frac{u_x^2(x)}
{1+u_{x}^2(x)+\bar\rho^2(x)}dx$.} 

Let $S=\{\xi\in\Real\mid y_\xi(\xi)=0\}$ and $S_n=\{\xi\in\Real\mid y_{n,\xi}(\xi)=0\}$. Furthermore, let $B_n=y_n(S)$. Then we claim that $\meas(B_n)\to 0$ as $n\to\infty$.

By definition, we have that $S=\{\xi\in\Real\mid y_\xi(\xi)=0\}=\{\xi\in\Real\mid h(\xi)=1\}$, which implies that $\meas(S)\leq \norm{h}_{L^1}$. Thus 
\begin{align*}
 \meas(B_n)=\meas(y_n(S))&=\int_{y_n(S)} dx\\ \nn 
& =\int_Sy_{n,\xi}(\xi)d\xi=\int_S \big(y_{n,\xi}(\xi)-y_\xi(\xi)\big)d\xi\\ \nn 
& \leq \meas(S)^{1/2}\norm{y_{n,\xi}-y_{\xi}}_{L^2}\\ \nn 
& \leq \norm{h}_{L^1}^{1/2}\norm{y_{n,\xi}-y_\xi}_{L^2}.
\end{align*}

By assumption $y_{n,\xi}-y_\xi\to 0$ in $L^2(\Real)$, and hence $\meas(B_n)$ tends to $0$ as $n\to\infty$. Moreover, 
\begin{equation*}
U_{n,\xi}(\xi)=0 \text{ for }\xi\in S_n \quad \text{ and }\quad U_\xi(\xi)=0 \text{ for }\xi\in S.
\end{equation*}
As far as $y_{n,\xi}(\xi)$ and $y_\xi(\xi)$ are concerned, we have the representations
\begin{equation*}
y_{n,\xi}(\xi)=\frac{1}{1+u_{n,x}^2(y_n(\xi))+\bar\rho_n^2(y_n(\xi))}\quad \text{ for almost every }\xi\in S_n^c
\end{equation*}
and 
\begin{equation*}
y_\xi(\xi)=\frac{1}{1+u_x^2(y(\xi))+\bar\rho^2(y(\xi))}\quad \text{ for almost every }\xi\in S^c.
\end{equation*}
This means, in particular, that 
\begin{align*}
 \int_{S^c} U_{n,\xi}^2(\xi)d\xi& =\int_{S^c\cap S^c_n} U_{n,\xi}^2(\xi)d\xi\\ \nn
 &=\int_{S^c\cap S^c_n} u_{n,x}^2(y_n(\xi))y_{n,\xi}^2(\xi)d\xi\\ \nn 
& =\int_{S^c\cap S^c_n} \frac{u_{n,x}^2(y_n(\xi))}{1+u_{n,x}^2(y_n(\xi))+\bar\rho_n^2(y_n(\xi))}y_{n,\xi}(\xi)d\xi\\ \nn 
& =\int_{B_n^c\cap y_n(S_n^c)}\frac{u_{n,x}^2(x)}{1+u_{n,x}^2(x)+\bar\rho_n^2(x)}dx\\ \nn
& =\int_{B_n^c}\frac{u_{n,x}^2(x)}{1+u_{n,x}^2(x)+\bar\rho_n^2(x)}dx,
\end{align*}
since $y_n$ is Lipschitz continuous and therefore
$\meas(y_n(S_n))=0$ for all $n\in\Natural$.
Similarly, one obtains
\begin{equation*}
 \int_{S^c}U_{\xi}^2(\xi)d\xi=\int_\Real \frac{u_x^2(x)}{1+u_x^2(x)+\bar\rho^2(x)}dx,
\end{equation*}
as $\meas(y(S))=0$.
Since $U_{n,\xi}\to U_{\xi}$ in $L^2(\Real)$ we find that
\begin{equation*}
 \lim_{n\to\infty}\int_{B_n^c}\frac{u_{n,x}^2(x)}{1+u_{n,x}^2(x)+\bar\rho_n^2(x)}dx=\int_\Real \frac{u_x^2(x)}{1+u_x^2(x)+\bar\rho^2(x)}dx.
\end{equation*}
Furthermore, note that $\frac{u_{n,x}^2(x)}{1+u_{n,x}^2(x)+\bar\rho_n^2(x))}$ is uniformly bounded by $1$ for all $n\in\Natural$ and $x\in\Real$. This means, in particular, that 
\begin{equation*}
 \int_{B_n} \frac{u_{n,x}^2(x)}{1+u_{n,x}^2(x)+\bar\rho_n^2(x)}dx\leq \meas (B_n), 
\end{equation*}
and thus the term on the left-hand side converges to $0$ as $n\to\infty$ since $\meas(B_n)\to 0$. 
Thus we get 
\begin{equation*}
 \lim_{n\to\infty}\int_{\Real}\frac{u_{n,x}^2(x)}{1+u_{n,x}^2(x)+\bar\rho_n^2(x)}dx=\int_\Real \frac{u_x^2(x)}{1+u_x^2(x)+\bar\rho^2(x)}dx.
\end{equation*}

{\bf Step 7. Convergence of $\int_\Real
  \frac{\bar\rho_{n}^2(x)}{1+u_{n,x}^2(x)+\bar\rho_n^2(x)}dx$ to
  $\int_\Real \frac{\bar \rho^2(x)}{1+u_x^2(x)+\bar\rho^2(x)}dx$.}
The argument is similar to the one in Step 6.

{\bf Step 8. Convergence of $k_n$ to $k$.} By definition we have 
\begin{equation*}
r(\xi)=\bar r(\xi)+ky_\xi(\xi)
\end{equation*}
and 
\begin{equation*}
r_n(\xi)=\bar r_n(\xi)+k_ny_{n,\xi}(\xi).
\end{equation*}
By assumption $X_n$ converges to $X$ in $E$, and thus according to \eqref{norm:E}, we infer that $k_n\to k$.

\end{proof}

\begin{remark}
Note that the convergence in Lagrangian coordinates implies in particular that $\bar\rho_n$ converges to $\bar\rho$ weakly.
Thus, in the special case $k=0$ and $r(\xi)=\bar r(\xi)=0$ for all $\xi\in\Real$, we infer that $\bar\rho_n$ converges weakly to zero and $k_n\to 0$. Thus $\rho(x)=0$ for all $x\in\Real$, and hence $(u,\rho,\mu)=(u,0,\mu)$ belongs to the set of Eulerian coordinates for the CH equation. Thus the sequence $\bar \rho_n$ in Theorem~\ref{thm:approxEuler} converges to zero in the weak sense, since all the assumptions in Theorem~\ref{thm:LagEuler} are satisfied due to Theorem~\ref{thm:approxLagran}.
\end{remark}

\section{Convergence in Eulerian coordinates implies convergence in Lagrangian coordinates} \label{sec:5}
In this section we want to show that convergence in Eulerian coordinates implies convergence in Lagrangian coordinates. Due to the definition of Eulerian coordinates, one might guess that it is natural to impose only the condition $u_n\to u$ in $H^1(\Real)$. However, due to Theorem~\ref{thm:LagEuler} we will require a somewhat stronger mode of convergence for $u_{n,x}$ to $u_x$, which in the end yields an equivalence between convergence in Eulerian and Lagrangian coordinates.

\begin{theorem}\label{thm:EulerLag}
Given a sequence $(u_n,\rho_n,\mu_n)\in \D$ and $(u,\rho,\mu)\in\D$
such that $(u_n,\rho_n,\mu_n)$ converges to $(u,\rho,\mu)$ as $n\to
\infty$ in the sense of Definition~\ref{def:convD},
let $X_n=(y_n,U_n,h_n,r_n)=L((u_n,\rho_n,\mu_n))\in \F_0$ and $X=(y,U,h,r)=L((u,\rho,\mu))\in \F_0$.  Then $X_n\to X$ in $E$ as $n\to \infty$.
\end{theorem}

\begin{proof}  The proof is divided into 7 steps.

{\bf Step 1. The sequence $y_n$ converges pointwise to $y$.}  
Denote
$D=\{\xi\in\Real \mid F \text{ is discontinuous at }y(\xi) \}$. By construction we have for all $\xi\in D^c$ that 
\begin{equation}\label{eq:help1}
F(y(\xi))+y(\xi)=\xi,
\end{equation}
and, in particular,
\begin{equation}\label{eq:help3}
F(y(\xi))=\mu((-\infty,y(\xi)])=\mu((-\infty, y(\xi))).
\end{equation}
As far as $y_n(\xi)$ is concerned, we have by \eqref{eq:Ldef1} that 
\begin{equation}\label{eq:help2}
\mu_n((-\infty, y_n(\xi)))+y_n(\xi)\leq \xi\leq \mu_n((-\infty, y_n(\xi)])+y_n(\xi)=F_n(y_n(\xi))+y_n(\xi).
\end{equation}
To show the pointwise convergence of $y_n(\xi)$ to $y(\xi)$ for $\xi\in D^c$, we have to distinguish two cases.

\noindent
For $y_n(\xi)\leq y(\xi)$, combining \eqref{eq:help1}--\eqref{eq:help2} yields
\begin{align*}
F(y(\xi))+y(\xi)=\xi&\leq F_n(y_n(\xi))+y_n(\xi) \\
& = F_n(y(\xi))-\mu_n((y_n(\xi),y(\xi)])+y_n(\xi).
\end{align*}
 Thus 
 \begin{equation*}
 0\leq y(\xi)-y_n(\xi)+\mu_n((y_n(\xi),y(\xi)])\leq F_n(y(\xi))-F(y(\xi)).
 \end{equation*}
 
 \noindent
 For $y(\xi)< y_n(\xi)$, combining again \eqref{eq:help1}--\eqref{eq:help2} yields
 \begin{align*}
 \mu_n((-\infty, y_n(\xi)))+y_n(\xi)
 & = F_n(y(\xi))+\mu_n((y(\xi),y_n(\xi)))+y_n(\xi)\\
 & \leq\xi=F(y(\xi))+y(\xi).
 \end{align*}
 Hence
 \begin{equation*}
 0\leq y_n(\xi)-y(\xi)+\mu_n((y(\xi), y_n(\xi)))\leq F(y(\xi))-F_n(y(\xi)).
 \end{equation*}

Since $\mu_n$ and $\mu$ are positive finite Radon measures for all $n\in \Natural$, we get that 
\begin{equation}\label{eq:new}
\vert y_n(\xi)-y(\xi)\vert \leq \vert F_n(y(\xi))-F(y(\xi))\vert, \quad \xi\in D^c.
\end{equation}
Since by assumption $\xi \in D^c$, we have that $F$ is continuous
at the point $y(\xi)$, which in turn implies that
$\vert F_n(y(\xi))-F(y(\xi))\vert$ converges to zero.
Thus $y_n(\xi) \to y(\xi)$ for every $\xi \in D^c$.

For $\xi\in D$, we argue as follows. Any point $x$ at which $F$ is discontinuous in Eulerian coordinates, corresponds to a maximal interval $[\xi_1,\xi_2]$ in Lagrangian coordinates such that $y(\xi)=x$ for all $\xi\in [\xi_1,\xi_2]$ and $\xi_2-\xi_1=\mu(\{x\})$. In particular, there exists an increasing sequence $\xi_i\in D^c$ such that $\xi_i$ converges to $\xi_1$. We may write
\begin{equation*}
y_n(\xi_1)-y(\xi_1)=(y_n(\xi_1)- y_n(\xi_i))+(y_n(\xi_i)-y(\xi_i))+(y(\xi_i)-y(\xi_1)).
\end{equation*}
Because $y_n$ and $y$ are Lipschitz continuous with Lipschitz constant
at most $1$ due to \eqref{eq:Ldef1}, 
we can thus estimate
\[
\abs{y_n(\xi_1)-y(\xi_1)} \le 2\abs{\xi_i-\xi_1} + \abs{y_n(\xi_i)-y(\xi_i)}.
\]
Since $y$ is continuous at $\xi_i$ (cf.~\eqref{eq:new}), the second
term on the right-hand side tends to 0 as $n \to \infty$, which shows
that $\abs{y_n(\xi_1)-y(\xi_1)}$ can be made arbitrarily small and
thus $y_n(\xi_1) \to y(\xi_1)$. A similar argument shows that
$y_n(\xi_2)\to y(\xi_2)$ by taking a decreasing sequence $\xi_i\in
D^c$ such that $\xi_i$ converges to $\xi_2$.

We can now show that $y_n(\xi)\to y(\xi)$ for all $\xi\in [\xi_1,\xi_2]$, By definition $y_n$ is an increasing function, and $y(\xi)$ is constant on $[\xi_1,\xi_2]$. Thus 
\begin{equation*}
| y_n(\xi)-y(\xi)| \leq\max \bigl(\lvert y_n(\xi_1)-y(\xi_1)\rvert, \lvert y_n(\xi_2)-y(\xi_2)\rvert\bigr)\quad \text{ for all } \xi\in[\xi_1,\xi_2]. 
\end{equation*}
Since both $|y_n(\xi_1)-y(\xi_1)|$ and $ |y_n(\xi_2)-y(\xi_2)|$ tend to zero as $n\to\infty$, it follows immediately that $y_n(\xi)\to y(\xi)$ for all $\xi\in [\xi_1,\xi_2]$. Thus $y_n\to y$ pointwise.

{\bf Step 2. Convergence of $h_n$ to $h$  and $\zeta_{n,\xi}$ to $\zeta_\xi$ in $L^2(\Real)$.}
By definition, we have that $X_n\in \F_0$ for all $n\in \Natural$ and $X\in \F_0$. Thus $H_n(\xi)=\xi-y_n(\xi)$, $n\in \Natural$, and $H(\xi)=\xi-y(\xi)$ for almost every $\xi\in\Real$. As $y_n(\xi)$ converges pointwise to $y(\xi)$, we infer that $H_n(\xi)\to H(\xi)$ pointwise almost everywhere as $n\to\infty$. Moreover, $H_n(\xi)$, $n\in\Natural$, and $H(\xi)$ are all continuous, and hence we conclude that, actually, we have pointwise convergence of $H_n(\xi)\to H(\xi)$ for every $\xi\in\Real$. Moreover, since 
\begin{equation*}
H_n(\xi)=\int_{-\infty}^\xi h_n(\eta)d\eta\quad\text{and}\quad H(\xi)=\int_{-\infty}^\xi h(\eta)d\eta, 
\end{equation*}
$h_n$ and $h$ can be seen as positive finite Radon measures, and hence 
\begin{equation}\label{est:weakh}
\lim_{n\to\infty} \int_\Real h_n(\xi)\psi(\xi)d\xi=\int_\Real h(\xi)\psi(\xi)d\xi
\end{equation}
for all $\psi\in C_c^\infty(\Real)$
according to \cite[Props.~7.19 and 8.17]{Folland}.  If we can show that $\norm{h_n}_{L^2}\to \norm{h}_{L^2}$, \eqref{est:weakh} will remain true for all $\psi\in L^2(\Real)$ by a density argument and hence all assumptions of the Radon--Riesz theorem are satisfied. Thus $h_n\to h$ in $L^2(\Real)$, provided $\norm{h_n}_{L^2}\to \norm{h}_{L^2}$.

In order to show this convergence, observe that 
\begin{equation}\label{eq:52}
\norm{h_n}_{L^1}=\norm{H_n}_{L^\infty}=\mu_n(\Real)\to \mu(\Real)=\norm{H}_{L^\infty}=\norm{h}_{L^1} \quad \text{as }n\to\infty.
\end{equation}
Since $X_n\in\F_0$ and $X\in\F_0$, we have because of \eqref{eq:lagcoord3} that
\begin{equation}\label{eq:53}
h_n^2(\xi)=h_n(\xi)-U_{n,\xi}^2(\xi)-\bar r_n^2(\xi)\quad \text{and}\quad h^2(\xi)=h(\xi)-U^2_\xi(\xi)-\bar r^2(\xi),
\end{equation}
respectively. Moreover, let $S=\{\xi\in\Real\mid y_\xi(\xi)=0\}$ and $S_n=\{\xi\in\Real\mid y_{n,\xi}(\xi)=0\}$.  Then 
\begin{equation*}
y_{n,\xi}(\xi)=\frac{1}{1+u_{n,x}^2(y_n(\xi))+\bar \rho_n^2(y_n(\xi))}\quad \text{ for almost every }\xi \in S_n^c,
\end{equation*}
and 
\begin{equation*}
y_\xi(\xi)=\frac{1}{1+u_x^2(y(\xi))+\bar \rho^2(y(\xi))}\quad \text{ for almost every }\xi\in S^c .
\end{equation*}
Hence we get 
\begin{align*}
 \int_\Real U^2_{n,\xi}(\xi)d\xi&=\int_{S_n^c} U^2_{n,\xi}(\xi)d\xi =\int_{S_n^c} \frac{u_{n,x}^2(y_n(\xi))}{1+u_{n,x}^2(y_n(\xi))+\bar \rho_n^2(y_n(\xi))}y_{n,\xi}(\xi)d\xi\\ \nn 
& =\int_\Real \frac{u_{n,x}^2(x)}{1+u_{n,x}^2(x)+\bar\rho_n^2(x)}dx,
\end{align*}
where we used that $U_{n,\xi}(\xi)=y_{n,\xi}(\xi)=0$ for all $\xi\in S_n$ and that $\meas (y_n(S_n))=0$. Similar arguments yield
\begin{equation*}
\int_\Real U^2_\xi(\xi)d\xi=\int_\Real \frac{u_x^2(x)}{1+u_x^2(x)+\bar\rho^2(x)}dx.
\end{equation*}
Thus, according to \eqref{assump:EulerLag1},
\begin{equation}\label{eq:51}
\lim_{n\to\infty}\norm{U_{n,\xi}}_{L^2}^2= \lim_{n\to\infty} \int_\Real U_{n,\xi}^2(\xi)d\xi=\int_\Real U_{\xi}^2(\xi)d\xi=\norm{U_\xi}_{L^2}^2.
\end{equation}
Following the same argument, this time using \eqref{assump:EulerLag2}, we obtain
\begin{equation}\label{eq:50}
\lim_{n\to\infty}\norm{\bar r_n}_{L^2}^2=\norm{\bar r}_{L^2}^2.
\end{equation}
Hence combining \eqref{eq:52}--\eqref{eq:53} and \eqref{eq:51}--\eqref{eq:50} yields that $\norm{h_n}_{L^2}\to \norm{h}_{L^2}$, and, in particular, $h_n\to h$ in $L^2(\Real)$ and $y_{n,\xi}\to y_\xi$ in $L^2(\Real)$, since both $X_n$ and $X$ belong to $\F_0$. 

{\bf Step 3. Convergence of $U_{n,\xi}$ to $U_\xi$ in $L^2(\Real)$.}
In order to conclude that $U_{n,\xi}\to U_\xi$ in $L^2(\Real)$, it suffices to show, according to the Radon--Riesz theorem, that $U_{n,\xi}(\xi)\rightharpoonup U_\xi(\xi)$ since we have convergence of the corresponding norms, cf.~\eqref{eq:51}. Due to the fact that $C_c^\infty(\Real)$ is dense in $L^2(\Real)$, it suffices to show that 
\begin{equation*}
\lim_{n\to\infty}\int_\Real U_{n,\xi}(\xi)\psi(\xi)d\xi=\int_\Real U_\xi(\xi)\psi(\xi)d\xi  
\end{equation*}
for all $\psi\in C_c^\infty(\Real)$.
Observe that we have for any $\psi\in C_c^\infty(\Real)$
\begin{equation*}
\int_\Real U_{n,\xi}(\xi)\psi(\xi)d\xi  = \int_{S_n^c} U_{n,\xi}(\xi)\psi(\xi)d\xi= \int_\Real u_{n,x}(x)\psi(y_n^{-1}(x))dx.
\end{equation*}
Here $S_n=\{\xi\in\Real\mid y_{n,\xi}(\xi)=0\}$ and hence, according to \eqref{eq:lagcoord3}, $U_{n,\xi}(\xi)=0$ for almost every $\xi\in S_n$, and $y_n^{-1}(x)$ denotes the pseudo inverse to $y_n(\xi)$ defined as
\begin{equation*}
y_n^{-1}(x)=\inf\{\xi\in\Real\mid y_n(\xi)>x\}.
\end{equation*}
Similarly, 
\begin{equation*}
\int_\Real U_\xi(\xi)\psi(\xi)d\xi=\int_\Real u_x(x)\psi(y^{-1}(x))dx,
\end{equation*}
where $y^{-1}(x)$ denotes the pseudo inverse to $y(\xi)$, i.e., 
\begin{equation*}
y^{-1}(x)=\inf\{\xi\in\Real\mid y(\xi)>x\}.
\end{equation*}
Thus it suffices to show that 
\begin{equation}\label{weak:unUn}
\lim_{n\to\infty} \int_\Real u_{n,x}(x)\psi(y_n^{-1}(x))dx=\int_\Real u_x(x)\psi(y^{-1}(x))dx
\end{equation}
for all $\psi\in C_c^\infty(\Real)$.

Let $\psi\in C_c^\infty(\Real)$. By assumption, there exist $c$, $d\in\Real$ such that $\supp(\psi)\subset [c,d]$. Then $\supp(\psi\circ y^{-1})\subset [\xi_1,\xi_2]$, where 
\begin{equation*}
\xi_1=\min\{\xi\in\Real \mid y(\xi)\geq c\} \quad \text{ and } \quad \xi_2=\max\{\xi\in\Real\mid y(\xi)\leq d\}.
\end{equation*}
Since $y(\xi)+H(\xi)=\xi$ for all $\xi\in\Real$ and $H(\xi)\leq \norm{\mu}$, we have 
\begin{equation*}
c-\norm{\mu}\leq \xi_1<\xi_2\leq d+\norm{\mu}
\end{equation*}
or, in other words, $\supp(\psi\circ y^{-1})\subset [c-\norm{\mu},d+\norm{\mu}]$.
Similarly one obtains that $\supp(\psi\circ y_n^{-1})\subset [c-\norm{\mu_n},d+\norm{\mu_n}]$. In particular, cf.~\eqref{assump:EulerLag3}, there exists $N\in\mathbb{N}$ such that 
\begin{equation*}
\supp(\psi\circ y_n^{-1})\subset [c-2\norm{\mu}, d+2\norm{\mu}], \quad \text{ for all } n\geq N.
\end{equation*}
Moreover, to any $x\in\Real$ we can assign a unique $y^{-1}(x)$ and $y^{-1}_n(x)$ using the pseudo inverse to $y(\xi)$ and $y_n(\xi)$, respectively.
Thus we have from \eqref{eq:Ldef1}
\begin{equation*}
y^{-1}(x)=x+\mu((-\infty, x])=x+F(x).
\end{equation*}
and
\begin{equation*}
y_n^{-1}(x)=x+\mu_n((-\infty,x])= x+F_n(x).
\end{equation*}
In particular, 
\begin{equation*}
\vert y_n^{-1}(x)-y^{-1}(x)\vert =\vert F_n(x)-F(x)\vert.
\end{equation*}
Thus $y_n^{-1}(x)$ converges to $y^{-1}(x)$ for any $x\in\Real$ at which $F(x)$ is continuous. In particular, $y_n^{-1}(x)$ converges to $y^{-1}(x)$ for almost every $x\in\Real$, since $F(x)$ has at most countably many discontinuities. Hence, after using Lebesgue's dominated convergence theorem, we obtain that $\psi\circ y_n^{-1}\to \psi\circ y^{-1}$ in $L^2(\Real)$.
Moreover, we have by assumption that $u_{n,x}$ converges weakly to $u_x$.
Thus $u_{n,x}\, \psi\circ y_n^{-1}$ is the product of a weakly convergent sequence
and a strongly convergent sequence, which implies that its integral
converges to the integral of the limit $u_x \,\psi\circ y^{-1}$,
which in turn proves~\eqref{weak:unUn}.

{\bf Step 4. Convergence of $\bar r_n$ to $\bar r$ in $ L^2(\Real)$.}
The proof follows exactly the same lines as the one of $U_{n,\xi}\to U_\xi$ in $L^2(\Real)$ in Step 3.

{\bf Step 5. Convergence of $\zeta_n$ to $\zeta$ in $L^\infty(\Real)$.}
Since both $X_n$ and $X$ belong to $\F_0$, we have 
\begin{equation}\label{eq:58}
\vert y(\xi)-y_n(\xi)\vert =\vert H(\xi)-H_n(\xi)\vert \leq \norm{h_n-h}_{L^1}.
\end{equation}
Moreover, by \eqref{eq:53},
\begin{equation*}
h_n(\xi)-h(\xi)=h_n^2(\xi)-h^2(\xi)+U_{n,\xi}^2(\xi)-U_\xi^2(\xi)+\bar r_n^2(\xi)-\bar r^2(\xi),
\end{equation*}
which together with \eqref{eq:58} implies that $\norm{y_n-y}_{L^\infty}\to 0$ as $n\to \infty$, since $h_n\to h$, $U_{n,\xi}\to U_\xi$ and $\bar r_n\to \bar r$ in $L^2(\Real)$.

{\bf Step 6. Convergence of $U_n$ to $U$ in $L^2(\Real)$.}
For a proof we refer the interested reader to \cite[Prop.~5.1]{HR}.

{\bf Step 7. Convergence of $k_n$ to $k$.}
By definition we have 
\begin{equation*}
r(\xi)=\bar r(\xi)+ky_\xi(\xi)=\bar\rho(y(\xi))y_\xi(\xi)+ky_\xi(\xi)=\rho(y(\xi))y_\xi
\end{equation*}
and 
\begin{equation*}
r_n(\xi)=\bar r_n(\xi)+k_ny_{n,\xi}(\xi)+\bar\rho_n(y_n(\xi))y_{n,\xi}(\xi)+k_ny_{n,\xi}(\xi)=\rho_n(y_n(\xi))y_{n,\xi}(\xi).
\end{equation*}
Thus, the constants $k_n$ in Eulerian and Lagrangian coordinates coincide and the same is true for $k$, and the claim is an immediate consequence of \eqref{assump:EulerLag4}.
\end{proof}

\section{Convergence for the initial data implies convergence for the solution of the 2CH system} \label{sec:6}
Finally, we would like to turn our attention to the 2CH system. In particular, we are going to study the consequences of the results derived so far in the context of the global weak conservative solutions of the 2CH system. 

\begin{theorem}\label{thm:timegeneralapprox}
Given $(u_0,\rho_0,\mu_0)$ in $\D$, let
$(u_{n,0},\rho_{n,0},\mu_{n,0})$ in $\D$ be such that
$(u_{n,0},\rho_{n,0},\mu_{n,0}) \to (u_0,\rho_0,\mu_0)$ in the sense
of Definition~\ref{def:convD}.
Consider the weak conservative solutions $(u(t,\dott),\rho(t,\dott),\mu(t,\dott))$ and $(u_n(t,\dott),\rho_n(t,\dott),\mu_n(t,\dott))$ of the 2CH system with initial data $(u,\rho,\mu)|_{t=0}=(u_0,\rho_0,\mu_0)$ and $(u_n,\rho_n,\mu_n)|_{t=0}=(u_{n,0},\rho_{n,0},\mu_{n,0})$, respectively. Then we have, for all $t\in[0,\infty)$,
\begin{align*}
u_n(t,\dott)&\to u(t,\dott) \in L^2(\Real)\cap L^\infty(\Real),\\
u_{n,x}(t,\dott)&\rightharpoonup u_x(t,\dott),\\
\bar\rho_n(t,\dott)&\rightharpoonup \bar\rho(t,\dott),  \\
k_n&\to k,\\
\int_\Real \frac{u_{n,x}^2(t,x)}{1+u_{n,x}^2(t,x)+\bar\rho_n^2(t,x)}dx&\to\int_\Real \frac{u_x^2(t,x)}{1+u_{x}^2(t,x)+\bar\rho^2(t,x)}dx,\\
\int_\Real \frac{\bar\rho_{n}^2(t,x)}{1+u_{n,x}^2(t,x)+\bar\rho_n^2(t,x)}dx&\to\int_\Real \frac{\bar\rho^2(t,x)}{1+u_x^2(t,x)+\bar\rho^2(t,x)}dx,\\
F_n(t,x)\to F(t,x)\text{ for every } &x \text{ at which } F(t,x)\text{ is continuous}, \\
F_n(t,\infty)& \to F(t,\infty),
\end{align*}
where $F_n(t,x)=\mu_n(t,(-\infty,x])$ for all $n\in \Natural$ and
$F(t,x)=\mu(t,(-\infty,x])$.
That is, for every $t \ge 0$ we have that the sequence
$(u_n(t,\dott),\rho_n(t,\dott),\mu_n(t,\dott))$ converges to
$(u(t,\dott),\rho(t,\dott),\mu(t,\dott))$ in the sense of
Definition~\ref{def:convD}.
\end{theorem}

\begin{proof}
Again, we are going to split the proof into several steps. 

{\bf Step 1. Convergence in Eulerian coordinates implies convergence in Lagrangian coordinates for the initial data.}
Let $X_0=(y_0,U_0,h_0,r_0)=L((u_0,\rho_0,\mu_0))$ and $X_{n,0}=(y_{n,0},U_{n,0},h_{n,0},r_{n,0})=L((u_{n,0},\rho_{n,0},\mu_{0,n}))$. Then according to Theorem~\ref{thm:EulerLag}, $X_{n,0}\to X_0$ in $E$.

{\bf Step 2. Convergence at initial time implies convergence at any later time for the solution in Lagrangian coordinates.}
Consider the following semilinear system of ordinary differential equations, which describes weak conservative solutions of the 2CH system in Lagrangian coordinates, cf.~\cite{GHR4},
\begin{subequations}\label{odesys:lag}
\begin{align}
\zeta_t& =U,\\
U_t&=-Q(X),\\
h_t& =2(U^2+\frac12 k^2-P)U_\xi,\\
\bar r_t&=-kU_\xi,\\
k_t&=0,
\end{align}
\end{subequations}
where $y(t,\xi)=\zeta(t,\xi)+\xi$, and 
\begin{align*}
P(X)(t,\xi)&=\frac14\int_\Real e^{-\vert y(t,\xi)-y(t,\eta)\vert}(2U^2y_\xi+2k\bar r+h)(t,\eta)d\eta+\frac12 k^2, \\
\intertext{and} 
Q(X)(t,\xi)&=-\frac14\int_\Real \sign(\xi-\eta)e^{-\vert y(t,\xi)-y(t,\eta)\vert}(2U^2y_\xi+2k\bar r+h)(t,\eta)d\eta.
\end{align*}
Then to $X(0)=X_0$ and $X_n(0)=X_{n,0}\in \F$, there exists a unique global solution to \eqref{odesys:lag} in $\F$, which we denote $X(t)$ and $X_n(t)$, respectively. Moreover, the mappings 
$X\mapsto P(X)-\frac12 k^2$ and $X\mapsto Q(X)$ are Lipschitz
continuous on bounded sets as mappings from $E$ to $H^1(\Real)$. In
particular, one has, cf.~\cite[Lemma 2.1]{HR}, that
\begin{equation}\label{Lip:P}
\norm{P(X_n(t))-P(X(t))}_{L^2\cap L^\infty}\leq C_{t,n}\norm{X_n(t)-X(t)}_{E},
\end{equation}
where $C_{t,n}$ is dependent on $\norm{X_n(t)}_{E}$, $\norm{X(t)}_E$. Similarly, we have 
\begin{equation}\label{Lip:Q}
\norm{Q(X_n(t))-Q(X(t))}_{L^2\cap L^\infty}\leq C_{t,n}\norm{X_{n}(t)-X(t)}_E.
\end{equation}
Furthermore, following closely the proof of \cite[Thm.~3.5]{GHR5} and \cite[Thm.~2.8]{HR}, we get from \eqref{odesys:lag} and \eqref{F0} that 
\begin{align}\nn 
\Sigma(X(t))& =\int_\Real (U^2y_\xi+h)(t,\eta)d\eta=\int_\Real (U^2y_\xi+h)(0,\eta)d\eta\\ \label{energy}
&=\int_\Real (U^2y_\xi+h^2+U_\xi^2+\bar r^2)(0,\eta)d\eta=\Sigma(X(0)),
\end{align}
and, in particular, 
\begin{equation*}
\norm{X(t)}_E\leq D(t,\Sigma(X(0)))\norm{X(0)}_E,
\end{equation*}
where $D(t,\Sigma(X(0)))$ depends on $t$ and $\Sigma(X(0))$. A similar
estimate holds for $X_n(t)$ with $n\in\Natural$. Thus $C_{t,n}$ in
\eqref{Lip:P} and \eqref{Lip:Q} only depends on $t$,
$\norm{X_n(0)}_E$, and $\norm{X(0)}_E$, due to
\eqref{energy}. Furthermore, since $X_n(0)\to X(0)$ in $E$, there
exists $M>0$ such that $\max(\norm{X_n(0)}_E,\norm{X(0)}_E)\leq M$.  
Thus, \eqref{Lip:P} and \eqref{Lip:Q} imply that the right-hand side of \eqref{odesys:lag} is Lipschitz continuous on bounded sets, and, in particular, applying Gronwall's inequality yields
\begin{equation}\label{Lip:compl}
\norm{X_n(t)-X(t)}_E\leq C_t\norm{X_n(0)-X(0)}_E,
\end{equation}
 where $C_t$ only depends on $M$ and $t$.
 
 {\bf Step 3. Convergence independent of relabeling in $\F$.} 
 As in \cite[Lemma 3.3]{HR}, one can show, given $T\geq 0$ and $\bar X(0)\in \F_0$, one has
 \begin{equation*}
 \bar X(t)\in \F_\alpha
 \end{equation*}
 for all $t\in[0,T]$, where $\alpha$ only depends on $t$ and
 $\norm{\bar X(0)}_E$. In our case, since $X_n(0)\to X(0)$ in $E$, 
 there exists $M > 0$ such that $M \ge \norm{X(0)}_E$ and $M\ge
 \norm{X_n(0)}_E$ for all $n \in \Natural$. Thus there exists
 $\beta(t)>0$ independent of $n$, such that
 \begin{equation*}
 X(t)\in \F_{\beta(t)} \quad \text{and} \quad X_n(t)\in \F_{\beta(t)}\quad  \text{ for all } n\in\Natural.
 \end{equation*} 
 Moreover, it is known, see, e.g., \cite[Lemma 4.6]{GHR4}, that for
 $\beta(t)\geq0$, the mapping $\Gamma\colon\F_{\beta(t)}\to \F_0$ with
 $X\mapsto \Gamma(X)= X\circ (y+H)^{-1}$ is continuous. Let $\tilde
 X_n(t)=\Gamma (X_n(t))$ and $\tilde X(t)=\Gamma (X(t))$. Then for each
 $t\geq 0$ the convergence $X_n(t)\to X(t)$ in $E$ implies $\tilde
 X_n(t)\to \tilde X(t)$ in $E$.
 
 {\bf Step 4. Convergence of the solutions  in Eulerian coordinates.}
 Since $\tilde X_n(t)\to \tilde X(t)$ in $E$ for all $t\geq 0$ and $\tilde X_n(t)$, $\tilde X(t)\in \F_0$ for all $t\geq 0$, applying Theorem~\ref{thm:LagEuler} finishes the proof.
\end{proof}
The next result gives the corresponding result in the case where the approximation is constructed using the mollifier. 

\begin{theorem}\label{thm:timesmoothapprox}
Given $(u_0,0,\mu_0)$ in $\D$  and let $(u_{n,0},\rho_{n,0}, \mu_{n,0})$ in $\D$ be the smooth approximation given by \eqref{smoothsequence} in Theorem~\ref{thm:approxEuler}. Consider the weak, conservative solutions  $(u(t,\dott),0,\mu(t,\dott))$ and $(u_n(t,\dott),\rho_n(t,\dott),\mu_n(t,\dott))$ of the 2CH system with initial data $(u,0,\mu)|_{t=0}=(u_0,0,\mu_0)$ and $(u_n,\rho_n,\mu_n)|_{t=0}=(u_{n,0},\rho_{n,0},\mu_{n,0})$, respectively. Then we have, for all $t\in[0,\infty)$,
\begin{align*}
u_n(t,\dott)&\to u(t,\dott) \in L^2(\Real)\cap L^\infty(\Real),\\
u_{n,x}(t,\dott)& \rightharpoonup u_x(t,\dott),\\
\bar\rho_n(t\dott) & \rightharpoonup \bar\rho(t,\dott),\\
k_n&\to 0,\\
\int_\Real \frac{u_{n,x}^2(t,x)}{1+u_{n,x}^2(t,x)+\bar\rho_n^2(t,x)}dx&\to\int_\Real \frac{u_x^2(t,x)}{1+u_{x}^2(t,x)}dx,\\
\int_\Real \frac{\bar\rho_{n}^2(t,x)}{1+u_{n,x}^2(t,x)+\bar\rho_n^2(t,x)}dx&\to0,\\
F_n(t,x)&\to F(t,x)\text{ for every } x \text{ at which } F(t,x)\text{ is continuous},\\
F_n(t,\infty)&\to F(t,\infty),
\end{align*}
where $F_n(t,x)=\mu_n(t,(-\infty,x])$ for all $n\in \Natural$ and $F(t,x)=\mu(t,(-\infty,x])$. 

Moreover, for each $n\in\Natural$, $(u_n(t,\dott),\rho_n(t,\dott),\mu_n(t,\dott))$ is a smooth solution to the 2CH system, that is, $u_n(t,x)$ and $\rho_n(t,x)$ belong to $C^\infty(\Real_{\ge 0} \times \Real)$ and $\mu_n(t,x)=\mu_{n,\rm ac}(t,x)=(u_x^2(t,x)+\bar\rho^2(t,x))dx$ for all  $t\geq 0$, and, in particular, no wave breaking occurs.
\end{theorem}

\begin{proof}
Since we showed in Theorem~\ref{thm:approxLagran} that
$X_n=(y_n,U_n,h_n,r_n)$ converges to $X=(y,U,h,0)$ in $E$, and hence
according to Theorem~\ref{thm:LagEuler}, the sequence
$(u_{n,0},\rho_{n,0},\mu_{n,0})$ converges to $(u_0,0,\mu_0)$ in the
sense of Definition~\ref{def:convD}, the first part of the theorem is an immediate consequence of 
Theorem~\ref{thm:timegeneralapprox}.

As far as the smoothness of the solution $(u_n(t,\dott),\rho_n(t,\dott),\mu_n(t,\dott))$ for any $t\geq 0$ and $n\in\Natural$ is concerned, we refer the interested reader to \cite[Sect.~6]{GHR4}.
\end{proof}


\begin{thebibliography}{99}

\bibitem{Ambrosio}
L. Ambrosio, N. Fusco, D. Pallara. 
\newblock {\em Functions of Bounded Variation and Free Discontinuity Problems}.
\newblock Oxford Mathematical Monographs, The Clarendon Press, Oxford University Press, New York, 2000.

\bibitem{BC} 
\newblock A. Bressan and A. Constantin.  
\newblock Global conservative solutions of the Camassa--Holm equation.  
\newblock {\em Arch. Ration. Mech. Anal.} 183:215--239, 2007.

\bibitem{BreCons:05a}
  A. Bressan and A. Constantin. 
\newblock  Global dissipative solutions of the Camassa--Holm equation.
\newblock  {\it Analysis and Applications} {\bf 5} (2007) 1--27. 


\bibitem{CH} 
\newblock R. Camassa and D.~D. Holm.  
\newblock An integrable shallow water equation with peaked solutions.  
\newblock {\em Phys. Rev. Lett} 71(11):1661--1664, 1993.

\bibitem{CHH} 
\newblock R. Camassa, D.~D. Holm, and J. Hyman.  
\newblock A new integrable shallow water equation.  
\newblock {\em Adv. Appl. Mech} 31:1--33, 1994.

\bibitem{ChenLiuZhang}
M. Chen, S.-Q. Liu and Y. Zhang.
\newblock A two-component generalization of the Camassa--Holm equation and its solutions.
\newblock {\em Lett. Math. Phys.}, 75:1--15, 2006.

\bibitem{ChenLiu}
R. M. Chen and Y. Liu.
\newblock Wave breaking and global existence for a generalized two-component Camassa--Holm system. 
\newblock {\em Inter. Math Research Notices}, Article ID rnq118, 36 pages, 2010.

\bibitem{cons:98b} 
\newblock A. Constantin and J. Escher.  
\newblock Wave breaking for nonlinear nonlocal shallow water equations.  
\newblock {\em Acta Math.} 181:229--243, 1998.

\bibitem{consIvan} 
\newblock A. Constantin, R.I. Ivanov.
\newblock On an integrable two-component Camassa--Holm shallow water system. 
\newblock {\em Phys. Lett. A}, 372:7129--7132, 2008.

\bibitem{EscherLechtenfeldYin}
J. Escher, O. Lechtenfeld, Z. Yin. 
\newblock Well-posedness and blow-up phenomena for the 2-component Camassa--Holm equation.
\newblock {\em Discrete Contin. Dyn. Syst.}, 19 (3):493--513, 2007.

\bibitem{Folland}
G.B. Folland.
\newblock  {\em Real Analysis}. 
\newblock  John Wiley and Sons Inc., New York,  2nd ed., 1999.

\bibitem{FuQu}
Y. Fu and C. Qu. 
\newblock Well posedness and blow-up solution for a new coupled Camassa--Holm
equations with peakons.
\newblock {\em J. Math. Phys.} 50:012906, 2009.

\bibitem{grunert}
K. Grunert.
\newblock Solutions of the Camassa--Holm equation with accumulating breaking times.
\newblock {\em Dyn. Partial Differ. Equ.}, 13:91--105, 2016.

\bibitem{GHR1}
K. Grunert, H. Holden, and X. Raynaud.
\newblock Lipschitz metric for the periodic Camassa--Holm equation. 
\newblock {\em J. Diff. Eq.}, 250:1460--1492, 2011. 

\bibitem{GHR2}
K. Grunert, H. Holden, and X. Raynaud.
\newblock Lipschitz metric for the Camassa--Holm equation on the line. 
\newblock  {\em Discrete Cont. Dyn. Syst., Series A}, 33, 2809--2827, 2013. 

\bibitem{GHR3}
K. Grunert, H. Holden, and X. Raynaud.
\newblock Global conservative solutions to the
  Camassa--Holm equation for initial data  with  nonvanishing asymptotics.
\newblock   {\em Discrete Cont. Dyn. Syst., Series A}, 32(12):4209--4227, 2012. 

\bibitem{GHR5}
  K. Grunert, H. Holden, and X. Raynaud.
\newblock Global dissipative solutions of the two-component Camassa--Holm system for initial data with nonvanishing asymptotics,
\newblock  {\it Nonlinear Anal. Real World Appl.} {\bf 17} (2014) 203--244.

\bibitem{GHR4}
K. Grunert, H. Holden, and X. Raynaud.
\newblock Global solutions for the two-component Camassa--Holm system.
\newblock  {\em Comm. Partial Differential Equations}, 37:2245--2271, 2012. 

\bibitem{GHR}
 K. Grunert, H. Holden, and X. Raynaud.
\newblock A continuous interpolation between conservative and dissipative solutions for the two-component Camassa--Holm system.
\newblock {\it Forum of Mathematics, Sigma}  vol. 1, e1, 70 pages doi:10.111, 2014.


\bibitem{GuanKarlsenYin}
C. Guan, K. H. Karlsen, and Z. Yin.
\newblock Well-posedness and blow-up phenomenal for a modified two-component Camassa--Holm equation.
\newblock In: {\em Nonlinear Partial Differential Equations and Hyperbolic Wave Phenomena}  (eds. H. Holden and K. H. Karlsen),  
Amer. Math. Soc., Contemporary Mathematics, vol. 526, pp. 199--220, 2010.

\bibitem{GY}
  \newblock C. Guan and Z. Yin.
  \newblock Global existence and blow-up phenomena for an integrable two-component Camassa--Holm water system.
  \newblock {\em J. Differential Equations}, 248:2003--2014, 2010. 




\bibitem{GuanYin2011} 
  C. Guan and Z. Yin.
  \newblock Global weak solutions for a  two-component Camassa--Holm shallow water system.
  \newblock {\em J. Func. Anal.} 260:1132--1154, 2011.


\bibitem{GuiLiu2010}  
  G. Gui and Y. Liu.
  \newblock On the global existence and wave breaking criteria for the two-component Camassa--Holm system.
  \newblock {\em J. Func. Anal.},  258:4251--4278, 2010.

\bibitem{GuiLiu2011}  
  G. Gui and Y. Liu.
  \newblock On the Cauchy problem for the two-component Camassa--Holm system.
  \newblock {\em Math Z},  268:45--66, 2011.


\bibitem{GuoZhou2010}
  Z. Guo and Y. Zhou.
  \newblock On solutions to a two-component generalized Camassa--Holm equation.
  \newblock {\em Studies Appl. Math.}, 124:307--322, 2010.



\bibitem{henry}
D. Henry. 
\newblock Infinite propagation speed for a two component Camassa--Holm equation. 
\newblock {\em Discrete Contin. Dyn. Syst. Ser. B}, 12 (3):597--606, 2009.

\bibitem{HR} 
H. Holden and X. Raynaud. 
\newblock Global conservative solutions for the Camassa--Holm equation --- a Lagrangian point of view.
\newblock {\em Comm. Partial Differential Equations} 32:1511--1549, 2007.

\bibitem{HRdiss} 
H. Holden and X. Raynaud. 
\newblock Dissipative solutions for the Camassa--Holm equation.
\newblock {\em Discrete Cont. Dyn. Syst.}, 24:1047--1112, 2009.

\bibitem{HRdissMP} 
H. Holden and X. Raynaud. 
\newblock Global dissipative multipeakon solutions for the Camassa--Holm equation.
\newblock   {\em Comm. Partial Differential Equations} 33:2040--2063, 2008.  

\bibitem{OlverRosenau}
\newblock P.~J. Olver and P. Rosenau.
\newblock Tri-hamiltonian duality between solitons and solitary-wave solutions having compact support.
\newblock {\em Phys. Rev. B}, 53(2):1900--1906, 1996.

\end{thebibliography}
\end{document}